\documentclass{amsart}
\usepackage{macros}
\title{A Callias-type index theorem with degenerate potentials}
\author{Chris Kottke}
\address{Northeastern University\\Department of Mathematics}
\email{c.kottke@neu.edu}
\date{\today}
\begin{document}
\begin{abstract}
A generalization of Callias' index theorem for self adjoint Dirac operators
with skew adjoint potentials on asymptotically conic manifolds is presented in
which the potential term may have constant rank nullspace at infinity. The
index obtained depends on the choice of a family of Fredholm extensions, though
as in the classical version it depends only on the data at infinity.
\end{abstract}
\maketitle

\section*{Introduction} \label{S:intro}
{\renewcommand{\theequation}{I.\arabic{equation}} 

In \cite{callias1978axial}, Callias proved an index theorem on $\R^{n}$, $n$
odd, for operators of the form $D + \Phi$, with $D$ a Dirac operator and $\Phi$
a skew-adjoint matrix-valued potential which is nondegenerate outside a compact
set.  A primary motivation was the consideration of spinors coupled to a
background {\em magnetic monopole}, a pair consisting of a connection form $A
\in \Omega^1(\R^3; \su(2))$ and a Higgs field $\Phi \in C^\infty(\R^3; \su(2))$
satisfying a certain partial differential equation. The spin Dirac operator
coupled to such a monopole has the required form since $\Phi$ is necessarily
nondegenerate outside a sufficiently large ball. 

Callias' index theorem was subsequently generalized in \cite{anghel1993index},
\cite{rade1994callias}, \cite{bunke1995k} and \cite{kottke2011index} to include
the case of arbitrary complete Riemannian manifolds, certain types of
pseudodifferential operators and other settings; however there has remained a
connection between Callias' index theorem and monopoles which has lacked a
rigorous treatment.  As noted originally in \cite{weinberg1979parameter},
the linearization of defining equation for monopoles about a solution $(A,\Phi)$
involves what appears to be a Callias-type operator, except that the potential
term is never invertible near infinity, so none of the existing Callias-type
index theorems apply.

In the present paper we remedy this situation, and present a generalization of
the Callias index theorem in which the potential may have nullspace of constant
rank at infinity. In \cite{kottke2013dimension}, the index theorem proved here
is used to calculate the virtual dimension of the monopole moduli space on
asymptotically conic 3-manifolds.

Here an {\em asymptotically conic}, or {\em scattering} manifold means a
manifold with boundary $X$ equipped with a metric of the form 
\[ 
	g = \frac{dx^2}{x^4} + \frac{h}{x^2}, 
\] 
where $x$ is a boundary defining function and $h$ is a symmetric 2-cotensor
restricting to a metric on $\pa X.$ On the interior of $X$, the metric is
smooth and complete; the standard definition of an asymptotically conic metric
appearing in the literature, in terms of a radial function $r$, is recovered by
setting $x = 1/r$.

The classical Callias index theorem on such an odd-dimensional manifold, for an
operator $D + \Phi$ acting on sections of a bundle $V \to X$, states that the
operator has a natural Fredholm extension and its index is given by
\begin{equation} 
	\ind(D + \Phi) = \ind(\pa^+_+), \quad \pa^+_+ \in \Diff(\pa X; V^+_+, V^-_+), 
\label{E:intro_classical_callias} 
\end{equation} 
where $\pa^+_+$
is the graded Dirac operator induced by $D$ at $\pa X$ acting on the positive
imaginary eigenspace bundle $V_+$ of $\Phi \rst_{\pa X},$ which by assumption
is invertible. 

Relaxing the condition that $\Phi$ be nondegenerate near $\pa X$ means the
operator no longer has a natural Fredholm extension on $L^2$, and so we make
use of {\em weighted} $L^2$-based Sobolev spaces instead, even allowing
different weights for sections near $\pa X$ with respect to the splitting $V
\rst_{\pa X} = V_0 \oplus V_1$ into the nullspace of $\Phi \rst_{\pa X}$ (which
we assume has constant rank) and its complement. The Fredholm weights are
determined by the invertibility of an operator family $I(\wt D_0,\lambda) \in
\Diff(\pa X; V_0)$, $\lambda \in \bbR$, which is obtained from $D$. The family
may be described briefly as the Mellin transform in the normal direction of the
formal expansion of the part of $D$ acting on $V_0$ at $\pa X$; see the
discussion below for more detail.

Furthermore, with the application to the monopole moduli space in mind, we
allow differential operators and potentials which are not strictly smooth up to
$\pa X$. Rather, we only require them to be {\em bounded polyhomogeneous},
meaning smooth on the interior of $X$, continuous up to $\pa X$, and having
complete asymptotic expansions at $\pa X$ in terms of real powers of $x$ and
positive integer powers of $\log x$.

\begin{thm-} 
Let $X$ be an asymptotically conic manifold of odd dimension and $P = D + \Phi$
a Callias-type operator (defined in \S \ref{S:fred}) whose potential $\Phi$ has
constant rank nullspace $V_0 \subset V \rst_{\pa X}$ at infinity. Then 
\begin{enumerate} 
[{\normalfont (a)}] 
\item
There exist bounded extensions 
\[ 
	P : \cH^{\alpha-1/2,\beta,k,1}(X; V) \to \cH^{\alpha+1/2,\beta,k,0}(X; V),\  
	\alpha,\beta \in \bbR,\ k \in \bbN_0 
\] 
of $P$ on a family of Sobolev spaces (see \eqref{E:intro_sob} below).  
\item 
The extensions are Fredholm if $\beta = \alpha + 1/2$ and $\alpha \notin
\bspec(\wt D_0)$, where $\bspec(\wt D_0) = \big\{\lambda : I(\wt D_0,\lambda)
\text{ is invertible}\big\}$ is a discrete set in $\bbR$. 
\item 
The index of such a Fredholm extension depends only on $\alpha$ and is
given by 
\[ 
	\ind(P,\alpha) = \ind(\pa^+_+) + \defect(P,\alpha) \in \bbZ.  
\]
Here $\pa^+_+$ is the graded Dirac operator on $V_+$ as above, 
and the {\em defect index} $\defect(P,\alpha)$, which also depends only on data
over $\pa X$, satisfies 
\[ 
	\defect(P,\alpha_0 - \varepsilon) = \defect(P,\alpha_0 + \varepsilon) + \dim \null\big(I(\wt D_0,\alpha_0)\big),
	\quad \alpha_0 \in \bspec(\wt D_0)
\] 
for $\varepsilon > 0$ sufficiently small. In addition, provided $D$ is
self-adjoint,
\[
	\defect(P,-\alpha) = -\defect(P,\alpha).
\]
\item 
Sections in the nullspace of such a Fredholm extension are smooth on the
interior and have complete asymptotic expansions at $\pa X$ with leading order
\[ 
	u_0 = \cO\big(x^{r + (\dim(X) -1)/2}(\log x)^{\ord(r)}\big), 
	\quad u_1 = \cO\big(x^{r + 1 + (\dim(X) - 1)/2}(\log x)^{\ord(r)}\big), 
\] 
where 
$\Null(P,\alpha) \ni u = u_0 \oplus u_1$ near $\pa X$ with respect to
$V\rst_{\pa X} = V_0 \oplus V_1$ and $r = \min\set{s \in \bspec(\wt D_0) : s >
\alpha}$.  
\end{enumerate} 
\end{thm-}

In case $\Phi \rst_{\pa X}$ is non-degenerate so $V_0$ vanishes, the Sobolev
spaces are independent of $\alpha$ and the index, which along with the
nullspace is independent of $\beta$ and $k$ in this case, reduces to the
classical version \eqref{E:intro_classical_callias}

To explain the above result in more detail, it is helpful to consider the
special case in which the vector bundle $V = V_0 \oplus V_1$ splits globally
over $X$ and the operator is diagonal, i.e.\ of the form 
\begin{equation} 
	P = \begin{pmatrix} D_0 & 0\\0 & D_1 + \Phi_1\end{pmatrix} 
	\in \scDiff^1(X; V_0\oplus V_1) 
\label{E:intro_P_globally_split} 
\end{equation} 
with $\Phi_1 \rst_{\pa X}$ nondegenerate. In this case the summands may be analyzed
separately. 
The operator $D_1 + \Phi_1$ is a conventional Callias type operator acting on
sections of $V_1$, with Fredholm extensions on the Sobolev spaces
$x^\beta\scH^l(X; V_1)$, where derivatives are taken with respect to vector
fields which are bounded with respect to $g$. The nullspace of $D_1 + \Phi_1$
and its adjoint consist of smooth sections which are rapidly decreasing at $\pa
X$ so the index is independent of $\beta$ and $l$.

On the other hand, the operator $D_0$ has a very different Fredholm theory,
which goes back to the work of Lockhart and McOwen in
\cite{lockhart1995elliptic} and also follows from the b-calculus of Melrose
\cite{melrose1993atiyah}. Indeed, $D_0$ does not have Fredholm extensions with
respect to Sobolev spaces above, but rather with respect to the weighted
b-Sobolev spaces $x^\alpha \bH^k(X; V_0)$ wherein derivatives are taken with
respect to vector fields which are bounded with respect to $x^2g$. The
extension 
\begin{equation} 
	D_0 :x^{\alpha-1/2} \bH^k(X; V_0) \to x^{\alpha+1/2}\bH^{k-1}(X; V_0) 
\label{E:intro_D0_extn} 
\end{equation} 
is Fredholm for all $\alpha \in \R \setminus B$ where $B$ is a discrete set. To
be precise, if we define $\wt D_0 = x^{-(n + 1)/2} D_0 x^{(n-1)/2}$, where $n =
\dim(X)$, then $B = \re \bspec(\wt D_0)$ is the set of {\em indicial roots} of
$\wt D_0$, meaning $\re(\lambda) \in \bbR$ for which the Mellin transformed
family $I(\wt D_0,\lambda) \in \Diff^1(\pa X; V_0)$ fails to be invertible.
Here the Mellin transform is defined in terms of the boundary defining function
$x$ on a product neighborhood of $\pa X$ by $ \cM(u)(\lambda) = \int_{0}^\infty
x^{-\lambda} u(x,y)\,\tfrac {dx} x$ and conjugation by $\cM$ on a differential
operator has the effect of formally replacing $x\pa_x$ by $\lambda \in \bbC$.
In the case at hand, $\bspec(\wt D_0)$ is purely real, and the index
of \eqref{E:intro_D0_extn} satisfies 
\[ 
\begin{aligned}
	\ind(D_0,-\alpha) &= -\ind(D_0,\alpha), \quad \text{if $D_0$ is self-adjoint, and}
	\\\ind(D_0,\alpha_0 - \varepsilon) &= \ind(D_0,\alpha_0 + \varepsilon) 
		+ \dim\Null\big(I(\wt D_0,\alpha_0)\big), \ \alpha_0 \in \bspec(\wt D_0), 
\end{aligned}
\] 
for $\varepsilon > 0$ sufficiently small.

It follows that in this special case \eqref{E:intro_P_globally_split} has
Fredholm extensions 
\[ 
\begin{gathered} 
	P : x^{\alpha-1/2}\bH^k(X; V_0) \oplus x^\beta \scH^l(X; V_1) 
	\to x^{\alpha+1/2}\bH^{k-1}(X; V_0)\oplus x^\beta \scH^{l-1}(X; V_1), \\ 
	\ind(P) = \ind(\pa^+_+) + \ind(D_0,\alpha),
\end{gathered} 
\] 
for $\alpha \notin \bspec(\wt D_0)$, and the Theorem above is essentially a
generalization of this result to the case where the splitting of $V$ at $\pa X$
does not extend globally and $D$ is only diagonal to leading order at $\pa X$.

In \S\ref{S:bkg} we recall some background material on b and scattering vector
fields, connections and Dirac operators, and we summarize the main results from
Melrose's theories of elliptic scattering operators \cite{melrose1994spectral}
and elliptic b-operators \cite{melrose1993atiyah}. Then in \S\ref{S:fred} we
define the precise class of Callias-type operators under consideration and the
hybrid Sobolev spaces $\cH^{\alpha,\beta,k,l}(X; V)$, which have the form
\begin{equation} 
	\cH^{\alpha,\beta,k,l}(X; V) \simeq 
		\begin{cases}x^\alpha\bH^{k+l}(X; V_0) \oplus x^\beta\bscH^{k,l}(X; V_1), &\text{near}\ \pa X, \\ 
		H^{k+l}(X; V) & \text{on}\ \mathring X, \end{cases}
\label{E:intro_sob} 
\end{equation} 
with respect to an extension of the splitting $V \rst_{\pa X} = V_0\oplus V_1$
to a neighborhood of $\pa X$. We prove parts (a), (b) and (d) of the
aforementioned result as Theorem~\ref{T:fredholm_hybrid_norealroot}. In
\S\ref{S:index} we compute the index and prove part (c) as
Theorem~\ref{T:hybrid_index_thm} by a deformation argument. We deform $\Phi$
continuously in a parameter $\tau$ to be nondegenerate at $\pa X$, and then a
family of Fredholm parametrices is constructed which is sufficiently uniform in
$\tau$ to allow for computation of the index. The parametrix family is
constructed using a special {\em transition calculus} of pseudodifferential
operators meant to capture the transition from Fredholm behavior like that of
$D_1 + \Phi_1$ to behavior like that of $D_0$ as the parameter $\tau \smallto
0.$ Parts of this calculus were first utilized by Guillarmou and Hassel in
\cite{guillarmou2008resolvent} to analyze the low energy limit of the resolvent
of the Laplacian on scattering manifolds, and we include a complete development
of the calculus as Appendix \ref{S:transition} in order to provide a general
reference.


\subsection*{Acknowledgments} The present paper represents part of the author's
PhD thesis work, and he is grateful to his thesis advisor Richard Melrose for
his support and guidance. The author has also benefited from numerous helpful
discussions with Michael Singer and Pierre Albin. Finally, he would also like
to thank the two referees for their careful reading, corrections, and general 
improvements to the readability of the manuscript.
}  

\section{Technical Background} \label{S:bkg} 

\subsection{Differential operators} \label{S:diffl_ops}

Let $X$ be an $n$ dimensional manifold with boundary, with a boundary defining
function $x,$ so $x \in C^\infty(X; [0,\infty))$ with
$x^{-1}(0) = \pa X$ and $dx \rst_{\pa X} \neq 0.$ Such a function is unique up to
multiplication by a smooth strictly positive function. The {\em b
vector fields} $\bV(X)$ are those smooth vector fields which are tangent to
$\pa X;$ they are characterized by the property
\begin{equation}
	\bV(X) \cdot x C^\infty(X) \subset xC^\infty(X).
	\label{E:b_vfields}
\end{equation}
The {\em scattering vector fields} $\scV(X)$ are defined by $\scV(X) =
x\bV(X).$ While $\bV(X)$ is naturally associated to $(X,\pa X)$, $\scV(X)$
depends on the choice on $x.$ As derivations on the algebra $C^\infty(X)$ these
satisfy
\begin{align}
\label{E:bb_comm} 
	{[\bV(X),\bV(X)]} &\subset \bV(X), \\
\label{E:bsc_comm}
	[\bV(X),\scV(X)] &\subset \scV(X)\ \text{and} \\
\label{E:scsc_comm}
	[\scV(X),\scV(X)] &\subset x \scV(X).
\end{align}
In other words, $\bV(X)$ and $\scV(X)$ are Lie subalgebras of the algebra
$\cV(X)$ of smooth vector fields and $\scV(X)$ is a $\bV(X)$-module. 

Near $\pa X$ these vector fields may be characterized in terms of local
coordinates. If $(x,y)$ form coordinates where the $y_i$ are
coordinates on $\pa X$ and $x$ is the boundary defining function then
\[
\begin{aligned}
	\bV(X) &\loceq \sspan_{C^\infty(X)}\set{x\pa_x, \pa_{y_i}} \\
	\scV(X) &\loceq \sspan_{C^\infty(X)}\set{x^2\pa_x, x\pa_{y_i}}
\end{aligned}
\]
Thus, over $C^\infty(X)$, $\bV(X)$ and $\scV(X)$ are locally free sheaves of rank $n =
\dim(X)$ and as such may be identified as the sections of well-defined
vector bundles, namely the {\em b tangent} and {\em scattering tangent}
bundles:
\[
\begin{aligned}
	\bV(X) &\equiv C^\infty(X; \bT X)  & \bT_p X &\loceq \sspan_\R \set{x \pa_x, \pa_{y_i}} \\
	\scV(X) &\equiv C^\infty(X; \scT X) & \scT_p X &\loceq \sspan_\R \set{x^2 \pa_x, x\pa_{y_i}}.
\end{aligned}
\]
The associated cotangent bundles $\bT^\ast X$, $\scT^\ast X$ are defined by
duality; they admit respective local frames $\set{dx/x, dy_i}$ and $\set{dx/x^2, dy_i/x}$.
There are well-defined bundle maps 
\begin{equation}
	\scT X \to \bT X \to T X
	\label{E:bkg_T_seq}
\end{equation}
induced by the inclusions $\scV(X) \subset \bV(X) \subset \cV(X)$ --- indeed,
these are just the obvious maps $x^2\pa_x \mapsto x(x\pa_x) \mapsto x^2(\pa_x)$
and $x\pa_{y_i} \mapsto x(\pa_{y_i})$ --- which, though isomorphisms
away from $\pa X$, are neither injective nor surjective at $\pa X.$ As
operators on $C^\infty(X)$, elements of $\bV(X)$ (resp.\ $\scV(X)$) may be
composed, resulting in {\em b} (resp.\ {\em scattering}) {\em differential operators}
\[
\begin{aligned}
	\bDiff^k(X) &= \bDiff^{k-1}(X) + \set{V_1\cdots V_k : V_i \in \bV(X)} \\
	\scDiff^k(X) &= \scDiff^{k-1}(X) + \set{V_1\cdots V_k : V_i \in \scV(X)}
\end{aligned}
\]
where $\bDiff^0(X) = \scDiff^0(X) = C^\infty(X).$ Alternatively, the (filtered)
algebra $\bDiff^\ast(X)$ may be considered as the universal enveloping algebra
of $\bV(X)$ over the ring $C^\infty(X)$ and likewise for $\scDiff^\ast(X).$
The definition of such operators acting on sections of vector bundles requires
the notion of a connection which is discussed next.

\subsection{Connections and Dirac operators} \label{S:bkg_cnxns}
Let $V \to X$ be a vector bundle, with associated principal (frame) bundle $\pi
: P \to X.$ As $P$ is also a manifold with boundary $\pa P = \pi^{-1}(\pa X)$
and defining function $\wt x = \pi^\ast(x)$, the b tangent bundle of
$P$ is well-defined, and $d\pi$ extends by continuity from the interior to
determine a short exact sequence
\begin{equation}
	0 \to V_pP \to \bT_p P \stackrel{d\pi} \to \bT_{\pi(p)} X \to 0 
	\label{E:princ_bundle_bseq}
\end{equation}
for each $p \in P,$ where $V_pP := \ker d\pi$ is isomorphic to the lie algebra
$\mathfrak{g}$ of the structure group $G$ of $P$ as in the case of closed
manifolds. 

Over the scattering tangent bundle of $X$, the correct structure on $P$ to
consider is the {\em fibered boundary} or $\phi$-tangent bundle $\phiT P$,
defined as in \cite{mazzeo1998pseudodifferential} by vector fields $V$ on $P$
such that $V\wt x = \wt x^2 C^\infty(P)$ and which are tangent at $\pa P$ to
the fibers of the boundary fibration $\phi = \pi : \pa P \to \pa X$. Locally, $\phiT P$ has
frames of the form $\set{x^2\pa_x, x\pa_{y_i}, \pa_{z_j}}$, where the $z_j$ are
coordinates in the fiber directions. Again $d\pi$ determines a short exact
sequence
\begin{equation}
	0 \to V_pP \to \phiT_p P \stackrel{d\pi} \to \scT_{\pi(p)} X \to 0
	\label{E:princ_bundle_phiseq}
\end{equation}

A {\em b (resp.\ scattering) connection} is
defined by one of three equivalent objects:
\begin{itemize}
\item An equivariant (with respect to the right action by $G$) choice of splitting $\bT_p P \cong \bT_{\pi(p)} X \oplus V_p P$ 
(resp.\ $\phiT_p P \cong \scT_{\pi(p)} X\oplus V_p P$) of the exact sequence
\eqref{E:princ_bundle_bseq} (resp.\ \eqref{E:princ_bundle_phiseq}) above.
\item A $\mathfrak{g}$ valued b (resp.\ $\phi$) one-form 
$\omega \in C^\infty(P; {}^{\mathrm{b}/\phi}T^\ast P\otimes \mathfrak{g})$
satisfying $\omega \rst_{VP\cong \mathfrak{g}} \equiv \id$ and transforming
equivariantly via $\omega_{p\cdot g} = \Ad(g^{-1})\cdot \omega_p$ for $g \in G$.
\item A covariant derivative operator 
$\nabla : C^\infty(X; V) \to C^\infty(X; {}^\mathrm{b/sc}T^\ast X \otimes V)$
satisfying $\nabla f\,s = df\, s + f \nabla s$ for $f \in C^\infty(X)$, $s \in
C^\infty(X; V)$. Note that $d : C^\infty(X) \to C^\infty({}^\mathrm{b/sc}T^\ast
X)$ is well-defined by continuity from the interior of $X$, and $\nabla$ is
related to $\omega$ locally via $\nabla = d + \gamma^\ast \omega$ with respect
to a local trivializing section $\gamma : U \subset X \to \pi^{-1}(U) \subset
P.$
\end{itemize} 
The duals of the natural maps $\phiT P \to \bT P \to TP$ and those in
\eqref{E:bkg_T_seq} give rise to maps
\[
\begin{aligned}
	C^\infty(P; T^\ast P) &\to C^\infty(P; \bT^\ast P) \to C^\infty(P; \phiT^\ast P) \\
	C^\infty(X; T^\ast X) &\to C^\infty(X; \bT^\ast X) \to C^\infty(P; \scT^\ast X),
\end{aligned}
\]
and we say that a b (resp.\ scattering) connection is the {\em lift of a true
(resp.\ b) connection} if its connection form $\omega \in C^\infty(P;
{}^{\mathrm{b}/\phi}T^\ast P\otimes \mathfrak{g})$ --- or equivalently each of the
local connection forms for $\nabla$ on $X$ --- is in the image of the
corresponding map. Observe that $d$, acting on functions, may be considered as the covariant
derivative associated to a b or scattering connection on the trivial line
bundle and in this sense is always the lift of a true connection on $X.$ 

In light of the inclusions $T\pa X \hookrightarrow TX \rst_{\pa X}$ and $T\pa X
\hookrightarrow \bT X \rst_{\pa X}$, true and b connections naturally induce
connections on $V \rst_{\pa X}$ by restriction.  In contrast, a scattering
connection generally does {\em not} induce such a connection on $V \rst_{\pa
X}$ unless it is the lift of a b connection.

With a connection on $V$ of the appropriate type, differential operators
$\bDiff^\ast(X; V)$ (resp.\ $\scDiff^\ast(X; V)$) may be defined as the universal
enveloping algebra of $\bV(X)$ (resp.\ $\scV(X)$) over $C^\infty(X; \End(V))$:
\[
\begin{aligned}
	\bDiff^k(X;V) &= \bDiff^{k-1}(X; V) + \set{\nabla_{V_1}\cdots \nabla_{V_k} : V_i \in \bV(X)}, \\
	\scDiff^k(X;V) &= \scDiff^{k-1}(X; V) + \set{\nabla_{V_1}\cdots \nabla_{V_k} : V_i \in \scV(X)}, \\
	\bDiff^0(X; V) &= \scDiff^0(X; V) = C^\infty(X; \End(V)).
\end{aligned}
\]
\begin{lem}
Suppose $P \in \scDiff^k(X; V)$ is defined as above in terms of a scattering
connection $\nabla.$ If $\nabla$ is the lift of a b-connection, then $P =
x^kP'$ where $P' \in \bDiff^k(X; V).$
\label{L:lift_cnxn}
\end{lem}
\begin{proof}
If $\nabla$ is the lift of a b-connection and $\scV(X) \ni V = xV'$
for $V' \in \bV(X)$ then the identity
\[
	\nabla_{V} = \nabla_{xV'} = x\nabla_{V'}
\]
holds. The claim then follows by induction, the derivation property of
$\nabla$ and the property \eqref{E:b_vfields}.
\end{proof}

Of particular importance are Dirac operators associated to a {\em scattering
metric}, meaning a metric of the form
\begin{equation}
	g = \frac{dx^2}{x^4} + \frac{h}{x^2},
	\label{E:exact_scattering_metric}
\end{equation}
in a collar neighborhood of $\pa X$ where the restriction $h\rst_{\pa X}$ is an
ordinary metric on $\pa X.$ By a result \cite{joshi1999recovering} of Joshi and
Sa Barreto, it is no loss of generality to assume that $h = h(x,y,dy)$, i.e.\
there are no cross terms of the form $\frac{dx}{x^2}\frac{dy}{x}.$ By
continuity then, $h(x)$ gives a family of nondegenerate metrics in a
neighborhood of $\pa X$.

For such a $g$, the scattering Clifford bundle $\scCl(X) = \Cl(\scT X, g)$ is
defined as in the case of a closed manifold, and given a Clifford module $V \to
X$ equipped with a scattering connection $\nabla$ which is compatible with the
Clifford action, one may define a {\em scattering Dirac operator}
\[
	D = \cl \circ \nabla \in \scDiff^1(X; V)
\]
by composing the connection $\nabla : C^\infty(X; V) \to C^\infty(X; \scT^\ast
X\otimes V)$ with the Clifford action $\cl : C^\infty(X; \scT^\ast \otimes V)
\to C^\infty(X; V).$ More generally, a {\em scattering Dirac-type operator}
is defined to be an operator differing from a scattering Dirac operator by a term
of order 0.

For a scattering metric, the Koszul formula shows that the Levi-Civita
connection is the lift of a b-connection\footnote{However it does {\em not} generally
coincide with the Levi-Civita connection for the conformally related b-metric
$x^2 g$ --- see the proof of Prop.\ 4.2 in \cite{kottke2013dimension}.}, and in
particular Lemma \ref{L:lift_cnxn} applies to the standard geometric operators
--- Dirac operators and Laplacians --- in this setting.

\subsection{Half-densities and Sobolev spaces} \label{S:bkg_dens}
Before discussing Fredholm results for elliptic scattering and b operators in
the next section, a word must be said about Sobolev spaces. First of all note
that a scattering metric such as \eqref{E:exact_scattering_metric} restricts to
a complete Riemannian metric on the interior of $X$, and likewise for a
b-metric (a positive definite element of $C^\infty\bpns{X; \Sym^2(\bT^\ast
X)}$). Along with a Hermitian structure on $V$, either of these may be used to
define $L^2(X; V)$ spaces; however the most invariant treatment involves
half-density bundles.

For any $s \in \R$, the $s$-density bundles $\bO^s(X) \to X$ and $\scO^s(X) \to
X$ are defined by
\[
\begin{aligned}
	\bO^s(X)_p &= \set{v : \Lambda^n(\bT_p X)\setminus 0 \to \R \;\big|\; v(t u) = \abs{t}^s v(u)}\\
	\scO^s(X)_p &= \set{v : \Lambda^n(\scT_p X)\setminus 0 \to \R \;\big|\; v(t u) = \abs{t}^s v(u)}
\end{aligned}
\]
For $s = 1$, there are invariantly defined integration maps
\[
	\int_X : C^\infty(X; \Omega_\mathrm{b/sc}(X)) \to \R \cup \pm \infty
\]
which are well-defined regardless of the orientability of $X$ since the density
bundles transform in accordance with the change of variables formula for the
integral.  Likewise, for $s = 1/2$ there is a natural bilinear pairing
\[
	\int_X : C^\infty\bpns{X; \Omega^{1/2}_\mathrm{b/sc}(X)} \times C^\infty\bpns{X; \Omega^{1/2}_\mathrm{b/sc}(X)} \to \R \cup \pm \infty
\]
which may be extended to $C^\infty\bpns{X; V\otimes \Omega^{1/2}_\mathrm{b/sc}(X)}$
whenever $V$ is equipped with a bilinear form. Taking the completion of
sections which are compactly supported in the interior of $X$ leads to the
natural Hilbert spaces $L^2(X; V\otimes\scOh)$ and $L^2(X; V\otimes \bOh)$.
We may freely pass from one density bundle to the other in light of the equivalence
\begin{equation}
	\scOh(X) \cong x^{-n/2} \bOh(X), \ \iff \ L^2(X; V\otimes \scOh) \cong x^{n/2}L^2(X; V\otimes \bOh).
	\label{E:bsc_L2}
\end{equation}

For $k \in \N$, the {\em b} (resp. {\em scattering}) {\em Sobolev spaces} are
defined to be the common maximal domains for b (resp. scattering) differential
operators of order $k$:
\[
\begin{aligned}
	\bH^k(X; V\otimes \Omega^{1/2}_\mathrm{b/sc}) &= \set{ u \in L^2(X; V\otimes \Omega^{1/2}_\mathrm{b/sc}) : P u \in L^2, \ \forall\; P \in \bDiff^k(X; V\otimes \Omega^{1/2}_\mathrm{b/sc})}\\
	\scH^k(X; V\otimes \Omega^{1/2}_\mathrm{b/sc}) &= \set{ u \in L^2(X; V\otimes \Omega^{1/2}_\mathrm{b/sc}) : P u \in L^2, \ \forall\; P \in \scDiff^k(X; V\otimes \Omega^{1/2}_\mathrm{b/sc})}
\end{aligned}
\]
Spaces of order $s \in \R$ may also be defined using either pseudodifferential
operators of the appropriate type or dual spaces and interpolation, though we
shall not require them. We will make use of weighted versions $x^\alpha
H^k_\mathrm{b/sc}$ of these spaces, along with the spaces
\[
\begin{aligned}
	\bscH^{k,l}(X; V\otimes \Omega^{1/2}_\mathrm{b/sc}) &= \set{u \in L^2 : D\circ P u \in L^2 \ \forall\; D \in \bDiff^k,\ P \in \scDiff^l}
	\\&= \set{u \in L^2 :  P\circ D u \in L^2 \ \forall\; D \in \bDiff^k,\ P \in \scDiff^l}.
\end{aligned}
\]
The equivalence of the two definitions here follows from \eqref{E:bsc_comm}. 

\begin{prop}
\mbox{}
\begin{enumerate}
[{\normalfont (a)}]
\item 
If $\alpha \geq \alpha'$, $k \geq k'$ and $l \geq l'$, then
\[
	x^\alpha \bscH^{k,l}(X; V\otimes \Omega^{1/2}_\mathrm{b/sc}) \subseteq x^{\alpha'} \bscH^{k',l'}(X; V\otimes \Omega^{1/2}_\mathrm{b/sc}).
\]
\item 
If in addition $\alpha > \alpha'$ and either $k > k'$ or $l > l'$, then the
inclusion is compact.
\item If $\alpha \geq \beta + l$, then
\[
	x^\alpha \bscH^{k,l}(X; V\otimes \Omega^{1/2}_\mathrm{b/sc}) \subseteq x^\beta \bH^{k+l}(X; V\otimes \Omega^{1/2}_\mathrm{b/sc}).
\]
\end{enumerate}
\label{P:compact_inclusion}
\end{prop}
\begin{proof}
(a) is straightforward and (c) follows from $\scV(X) =
x\bV(X).$ For (b), suppose that $\set{u_i}$ is a bounded sequence
in $x^{\alpha}\bscH^{k,l}$ --- in particular $\set{u_i}$ is bounded in
$x^{\alpha'}\bscH^{k',l'}$ as well, say by $M \in \bbR$. We show it has a
Cauchy subsequence in the latter space. Let $\set{K_j : j \in \bbN}$ be an
increasing sequence of compact sets in $X$ such that $X \setminus K_j = \set{x
< 1/j}$. The extra regularity implies that $\set{\chi_j u_i : i \in \bbN}$ has
a Cauchy subsequence in $x^{\alpha'}\bscH^{k',l'}$ for each $j$, where $\chi_j$
is a smooth cutoff supported in $K_j$ and which is equal to 1 on $K_{j-1}$.
Iteratively passing to these subsequences and then taking a diagonal
subsequence if necessary, we may assume that $\set{u_i}$ is Cauchy on each
$K_j$. Then for $N$ large and $j$ large enough that
$\pns{\tfrac{1}{j-1}}^{2(\alpha - \alpha')} < \varepsilon/2M$, for $n,m \geq
N$,
\begin{multline*}
	\norm{u_n - u_m}_{x^{\alpha'}\bscH^{k',l'}}^2
	  \leq \norm{(1 - \chi_j)(u_n - u_m)}^2_{x^{\alpha'}\bscH^{k',l'}} 
		+ \norm{\chi_j(u_n - u_m)}^2_{x^{\alpha'}\bscH^{k',l'}}
	\\ \leq \sup_{X\setminus K_{j-1}}(x^{\alpha - \alpha'})^2
	    \norm{(1 - \chi_j)(u_n - u_m)}^2_{x^{\alpha}\bscH^{k',l'}} + \varepsilon 
	  \leq 2\varepsilon,
\end{multline*}
so $\set{u_i}$ has a Cauchy subsequence in $x^{\alpha'}\bscH^{k',l'}$.  
\end{proof}

Finally, if $X$ {\em is} equipped with a scattering (respectively b) metric
$g$, then there is a canonical trivializing section of $\scO^s(X)$ (resp.\
$\bO^s(X)$) given by $\abs{d\mathrm{Vol}_g}^s$ which may be used to promote
ordinary sections of $V$ to half-densities, and the resulting $L^2$ space is
the same as the one defined in terms of $g$ described in the beginning of this
section. Differential operators may be lifted to act on half-density sections
by utilizing a connection with respect to which the canonical section
$\abs{d\mathrm{Vol}_g}^s$ is covariant constant.

\subsection{The scattering calculus} \label{S:bkg_scat}
We briefly review the Fredholm theory of scattering operators. We refer
the reader to \cite{melrose1994spectral} for proofs of the results here.

For $P \in \scDiff^k(X; V)$, two kinds of symbols are defined. The first is an
version $\sigma_k(P) \in C^\infty(\scT^\ast X; \End(V))$ of the usual principal symbol for differential
operators, given locally by
\begin{equation}
	\sigma_{k} : \sum_{j +\abs{\alpha}\leq k} a_{j,\alpha}(x,y)(x^2\pa_x)^j(x\pa_y)^\alpha 
	\mapsto \sum_{j + \abs{\alpha} = k} a_{j,\alpha}(x,y)(i\xi)^j (i\eta)^\alpha
	\label{E:sc_int_sym}
\end{equation}

The second is the fiberwise Fourier transform with respect to $\scT X\rst_{\pa
X} \to \pa X$ of the `restriction' of $P$ to the boundary in the following
sense: It follows from \eqref{E:scsc_comm} that the quotient Lie algebra
$\scV(X)/x\, \scV(X)$ is abelian, and this quotient map along with the
restriction $\rst_{\pa X} : C^\infty(X; \End(V)) \to C^\infty(\pa X; \End(V))$
leads to the restriction of $P$ to $\pa X$, denoted by $N_\fsc(P)$, which may
be viewed as a fiberwise differential operator $N_\fsc(P) \in
\Diff^k_{\mathrm{fib},I}(\scT X \rst_{\pa X}, V)$ which is translation
invariant with respect to the vector bundle structure (hence the subscript
$I$).  Locally,
\[
	N_\fsc : \sum_{j +\abs{\alpha}\leq k} a_{j,\alpha}(x,y)(x^2\pa_x)^j(x\pa_y)^\alpha 
	\mapsto \sum_{j + \abs{\alpha} \leq k} a_{j,\alpha}(0,y)(\pa_z)^j (\pa_w)^\alpha
\]
where $(z,w)$ are fiber coordinates for $\scT X\rst_{\pa X}.$ Taking the
Fourier transform gives the {\em scattering or boundary symbol} $\ssym(P) \in
C^\infty\bpns{\scT^\ast X\rst_{\pa X}; \End(V)}$ which is a (generally inhomogeneous)
polynomial of degree $k$ along the fibers:
\begin{equation}
	\ssym : \sum_{j +\abs{\alpha}\leq k} a_{j,\alpha}(x,y)(x^2\pa_x)^j(x\pa_y)^\alpha 
	\mapsto \sum_{j + \abs{\alpha} \leq k} a_{j,\alpha}(0,y)(i\xi)^j (i\eta)^\alpha.
	\label{E:sc_bound_sym}
\end{equation}
Note that \eqref{E:sc_int_sym} involves only the highest order part of the
operator while \eqref{E:sc_bound_sym} involves the lower orders as well. They
are compatible in that $\sigma_{k}(P)$ over $\pa X$ captures the leading
order asymptotic behavior of $\ssym(P)$ as $\abs{(\xi,\eta)} \to \infty.$

The main Fredholm and elliptic regularity results from the 
theory are summarized by
\begin{thm}
$P \in \scDiff^k(X; V\otimes\scOh)$ is said to be {\em fully elliptic} if
$\sigma_{k}(P)$ is invertible off the zero section and $\ssym(P)$ is
invertible everywhere. In this case $P$ admits Fredholm extensions
\[
	P : x^\alpha \scH^s(X; V\otimes \scOh) \to x^\alpha \scH^{s-k}(X; V\otimes \scOh)
\]
for all $\alpha,s \in \R$ with index independent of $\alpha$ and $s$. 
Furthermore
\[
	\Null(P) \subset x^\infty C^\infty(X; V\otimes \scOh)
\]
where $x^\infty C^\infty$ denotes smooth functions vanishing to infinite order
at $\pa X.$
\label{T:sc_thm}
\end{thm}

Using \eqref{E:bsc_comm} the result holds as well for the Fredholm extensions
\[
	P : x^\alpha \bscH^{r,s}(V; V\otimes \scOh) \to x^\alpha \bscH^{r,s-k}(X; V\otimes \scOh).
\]

The theorem above is a consequence of a parametrix construction within a calculus
of {\em scattering pseudodifferential operators} 
\[
	\scP^\ast(X; V) = \bigcup_{\substack{t \in \R \\e \in \Z}} \scP^{t,e}(X; V),
\]
which extend the differential operators above.
These are defined by their Schwartz kernels on the {\em scattering double
space}, $\scX^2$, a blown-up version of $X^2.$ An
operator $Q \in \scP^{t,e}(X; V)$ is conormal of order $t$ with respect to the
(lifted) diagonal, has Laurent expansion with leading order $e$ at the unique
boundary face (conventionally denoted $\fsc$) meeting the diagonal, and the coefficient of the
leading order term restricts to a conormal distribution with respect to the
zero section of a natural vector bundle structure $\mathring \fsc \cong \scT X
\rst_{\pa X} \to \pa X$, whose fiberwise Fourier transform is therefore a
fiberwise (total) symbol $\ssym(Q) \in C^\infty(\scT^\ast X \rst_{\pa X};
\End(V))$. 

Operators compose according to $\scP^{t,e}(X; V)\circ \scP^{s,f}(X;V) \subset
\scP^{t+s,e+f}(X;V)$, with symbols composing via
\[
\begin{aligned}
	\sigma_{t+s}(Q \circ P) &= \sigma_t(Q)\sigma_s(P) \\
	(\ssym)_{e+f}(Q \circ P) &= (\ssym)_e(Q) (\ssym)_f(P).
\end{aligned}
\]

$Q \in \scP^{t,e}(X; V)$ is bounded as an operator $x^\alpha\scH^s \to
x^{\alpha'}\scH^{s'}$ if $s' \leq s - t$ and $\alpha' \leq \alpha + e$, and
hence by Proposition~\ref{P:compact_inclusion} $Q$ is compact if strict
inequality holds. Furthermore, $Q$ is trace-class on $x^\alpha \scH^s$ provided
$t < -\dim(X)$ and $e > \dim(X)/2$. Theorem~\ref{T:sc_thm} follows from the
construction of a $Q \in \scP^{-k,0}(X; V)$ with $\sigma_{-k}(Q) =
\sigma_k(P)^{-1}$ and $\ssym(Q) = \ssym(P)^{-1}$, whence $PQ - I, QP - I \in
\scP^{-1,1}(X;V)$ are compact.

\subsection{The b calculus} \label{S:bkg_b_calc}
Finally we summarize the Fredholm theory for b operators, which is somewhat
different from the above. Proofs of the results in this section can be found in
\cite{melrose1993atiyah}.

Given $P \in \bDiff^k(X; V)$ the principal symbol $\sigma_k(P) \in
C^\infty\bpns{\bT^\ast X; \End(V)}$ is well-defined. In local coordinates
\[
	\sigma_k : \sum_{j + \abs \alpha \leq k} a_{j,\alpha}(x,y)(x\pa_x)^j\pa_y^\alpha \mapsto 
	\sum_{j + \abs \alpha = k} a_{j,\alpha}(x,y)(i\xi)^j(i\eta)^\alpha.
\]

In contrast to the scattering setup, the Lie algebra $\bV(X)/x\, \bV(X)$
induced over $\pa X$ is {\em not} abelian, though the element $x\pa_x$ descends
to have trivial bracket with all other elements. The induced restriction of $P$
to $\pa X$, here denoted $I(P)$ to match the usual convention, therefore
defines an operator
\[
	I(P) \in \Diff_I^k(\bN_+ \pa X; V)
\]
where $\bN_+ \pa X \to \pa X$ is the $\R_+ = [0,\infty)$ bundle
spanned\footnote{In fact $\bN_+ \pa X$ is well-defined independent of the
choice of $x$ as the span of the inward pointing sections of the kernel of the
bundle map $\bT X \to T X$ over $\pa X.$ It is trivial (being an $\R_+$
bundle), but not canonically so: the trivialization depends on a choice of
$x$.} by $x\pa_x$, and the $I$ denotes invariance with respect to the
multiplicative structure on $\bN_+ \pa X\cong \pa X \times \R_+$. Locally,
\[
	I : \sum_{j + \abs \alpha \leq k} a_{j,\alpha}(x,y)(x\pa_x)^j\pa_y^\alpha \mapsto 
	\sum_{j + \abs \alpha \leq k} a_{j,\alpha}(0,y)(s\pa_s)^j\pa_y^\alpha.
\]
Here $s \in \R_+$ denotes the fiber variable for $\bN_+ \pa X$. Conjugation by
the Mellin transform in $s$ gives the {\em indicial family} $I(P,\lambda) =
\cM I(P)\cM^{-1}$, where 
\begin{equation}
	\cM(u) = \int_{\R_+} s^{-\lambda} u(s) \frac{ds}{s}
	\label{E:mellin_xform}
\end{equation}
depends on the choice of $x$ used to trivialize $\bN_+ \pa X \cong \pa X \times
\R_+$. (Note that the convention \eqref{E:mellin_xform} differs from
\cite{melrose1993atiyah}, hence so does the convention for the b spectrum
below.) Locally, this just amounts to replacing $x\pa_x$ by $\lambda$ and
evaluating at the boundary:
\[
	I(\cdot,\lambda) : \sum_{j + \abs \alpha \leq k} a_{j,\alpha}(x,y)(x\pa_x)^j\pa_y^\alpha \mapsto 
	\sum_{j + \abs \alpha \leq k} a_{j,\alpha}(0,y)(\lambda)^j\pa_y^\alpha.
\]

$I(P,\lambda)$ is a holomorphic family with respect to $\lambda \in \C$ of
differential operators on $\pa X$ which are elliptic if $P$ is elliptic, which
we assume from now on.  It can be shown that the set
\[
	\bspec(P) = \set{\lambda \in \C : I(P,\lambda) \text{ not invertible}},
\]
called the {\em b-spectrum} of $P$, is discrete, satisfies $\abs{\re\, \lambda_j}
\smallto \infty$ as $j \smallto \infty$ and does not depend on the choice of $x.$
Equivalently,
\[
	I(P,\lambda)^{-1} \in \Psi^{-k}(\pa X; V)
\]
is a meromorphic family, with (finite order) poles at $\lambda \in \bspec(P).$ 

\begin{thm}
If $P \in \bDiff^k(X; V\otimes\bOh)$ is elliptic, then $P$ admits Fredholm
extensions
\begin{equation}
	P_\alpha : x^\alpha \bH^s(X; V\otimes \bOh) \to x^\alpha \bH^{s-k}(X; V\otimes \bOh)
	\label{E:b_fredholm_extn}
\end{equation}
for all $\alpha \notin \re\, \bspec(P)$, $s \in \R.$ Elements of
$\Null(P_\alpha)$ are smooth on the interior of $X$ and have asymptotic
expansions at $\pa X$ of the form
\begin{equation}
\begin{gathered}
	\Null(P_\alpha) \subset \cAphg^{\wh E^+(\alpha)}(X; V\otimes \bOh),
	\\ E^+(\alpha) = \set{(z+n,k) : z \in \bspec(P), \re z > \alpha, k = \ord(z), n \in \bbN},
	\\ \wh E^+(\alpha) = \ol \bigcup_k \big(E^+(\alpha) + k\big).
\end{gathered}
	\label{E:b_phg_expn}
\end{equation}
(See Appendix~\ref{S:phg}.) Here $\mathrm{ord}(z)$ denotes the order of the
pole of $I(P,\lambda)^{-1}$ at $\lambda = z$.
\label{T:b_fredholm}
\end{thm}
\noindent
The index of the Fredholm extension does not depend on $s$, but
does depend on $\alpha$ in a way we describe below.

The theorem follows by a parametrix construction within the {\em b calculus} of
pseudodifferential operators $\bP^\ast(X;V) = \bigcup \bP^{t,\cE}(X;V)$ with
Schwartz kernels on the {\em b double space} $\beta_\mathrm{b} : \bX^2 = [X^2;
\pa X^2] \to X^2$, which are conormal (of order $s$) to the diagonal and have
polyhomogeneous expansions (see Appendix \ref{S:phg}) given by the
index sets $\cE = (E_\fff, E_\flb, E_\frb)$, at the boundary faces of $\bX^2.$ 

For $Q \in \bP^{t,\cE}(X; V)$, the restriction of the leading order term in the
expansion of $Q$ at the boundary face $\fff$ meeting the diagonal is denoted
$I(Q)$. This has an interpretation as a pseudodifferential operator $I(Q) \in
\Psi_I^s(\bN_+ \pa X; V)$ of convolution type in the fiber directions, such
that with respect to composition $\bP^{t,\cE}(X;V)\times \bP^{s,\cF}(X;V) \ni (Q,P)
\mapsto Q \circ P \in \bP^{t+s,\cG}(X;V)$, 
\[
\begin{aligned}
	\sigma_{t+s}(Q\circ P) &= \sigma_t(Q)\sigma_s(P) \\
	I(Q\circ P) &= I(Q)I(P).
\end{aligned}
\]
Here $\cG = (G_\fff, G_\flb, G_\frb)$ where
\[
\begin{aligned}
	G_\flb & = (E_\fff + F_\flb) \ol \cup E_\flb & G_\frb &= (E_\frb + F_\fff) \ol\cup F_\frb \\
	G_\fff & = (E_\fff + F_\fff) \ol \cup (E_\flb + F_\frb) 
\end{aligned}
\]
in terms of the extended union of index sets (see Appendix \ref{S:phg}), and there is
a necessary condition $\re E_\frb + \re F_\flb > 0$ for the
composition to be defined (hence the term `calculus' as opposed to `algebra' of
operators).

$Q \in \bP^{t,\cE}(X;V)$ is bounded as an operator $x^\alpha \bH^s \to
x^{\alpha'} \bH^{s'}$ provided $s' \leq s - t$, $\re E_\frb > -\alpha$, $\re
E_\flb > \alpha'$ and $\re E_\fff \geq \alpha' - \alpha$, and it is compact (by
Proposition~\ref{P:compact_inclusion}) if strict inequality holds. $Q$ is trace
class on $x^\alpha \bH^s$ if the conditions for compactness hold, and in
addition $t < - \dim(X).$

Theorem~\ref{T:b_fredholm} is a consequence of constructing a parametrix $Q \in
\bP^{-k,\cE(\alpha)}(X; V)$ such that $\sigma_{-k}(Q) = \sigma_k(P)^{-1}$ and
$I(Q) = \cM^{-1}_{\re \lambda = \alpha} I(P,\lambda)^{-1}$, for which $PQ - I$
and $QP - I$ are compact. (In fact, they are trace-class, which will be
important later.) The characterization of the nullspace requires additional
analysis we will not describe here.

We next discuss the dependence of the index of \eqref{E:b_fredholm_extn} on
$\alpha.$ Each pole of $I(P,\lambda)^{-1}$, say at $\lambda = z$ for $z \in
\C$, is associated with a set of elements of the nullspace of $I(P)$ called the
{\em formal nullspace at $\lambda$}
\begin{equation}
	F'(P,\lambda) = \Big\{u = \sum_{0 \leq l < \mathrm{ord}(z)} s^{z}(\log s)^lu'(y) : I(P) u = 0\Big\}.
	\label{E:formal_nullspace}
\end{equation}
which may be identified with $\mathrm{ord}(z)$ copies of
$\Null\bpns{I(P,z)} \subset C^\infty(\pa X; V\otimes \Omega^{1/2})$. The
leading term in the asymptotic expansion of any $v \in \Null(P_\alpha)$ is in
$F'(P,z)$ for some $z \in \bspec(P)$ with $\re z > \alpha$. Using the
notation
\[
	F(P,r) = \bigoplus_{\re z = r} F'(P,z), \quad r \in \R
\]
and
\begin{multline}
	\Null(P,r) = \\\set{ v \sim x^{z}(\log x)^lv' + o(x^{z}(\log x)^l) : Pv = 0,\ z \in \bspec(P),\ \re(z) = r}
	\label{E:b_nullspace_wght}
\end{multline}
(so that with this notation $\Null(P_\alpha) = \bigcup_{r > \alpha}
\Null(P,r)$), we denote the image of $\Null(P,r)$ in $F(P,r)$ under this leading
term map by
\begin{equation}
	G(P,r) = \mathrm{Image}\bpns{\Null(P,r)\to F(P,r)} \cong \Null(P,r)/\Null(P,r')
	\label{E:formal_nullspace_image}
\end{equation}
where $r' = \min \set{\re z > r : z \in \bspec(P)}.$ 

It is a fundamental result that there is a nondegenerate bilinear pairing
(essentially a generalized Green's formula)
\begin{equation}
\begin{aligned}
	B &: F(P,r) \times F(P^\ast,-r) \to \C \\
	B(u,v) &= \frac 1 i \int_{\pa X \times \R_+} \pair{I(P)\phi u,\phi v} - \pair{\phi u,I(P^\ast)\phi v}
\end{aligned}
	\label{E:boundary_pairing}
\end{equation}
where $\phi \in C^\infty(\R_+)$ is a compactly supported cutoff with $\phi
\equiv 1$ in a neighborhood of $0$, on which $B$ does not depend.
Moreover,
\[
	G(P,r) = G(P^\ast,-r)^\perp
\]
with respect to $B$, which leads to the so-called {\em relative index theorem}:

\begin{thm}
The index $\ind(P_\alpha)$ of the extensions \eqref{E:b_fredholm_extn} is constant 
for $\alpha$ in connected components of $\R \setminus \re \bspec(P),$ 
and for $\alpha < \beta$, 
\[
	\ind(P_\alpha) - \ind(P_\beta) = \sum_{\alpha < r < \beta} \dim F(P,r).
\]
\label{T:b_relindex}
\end{thm}

Of particular interest are self-adjoint operators, for which $\bspec(P)$ is
purely real.
In this case, the identity $\ind(P_\alpha) = -\ind(P_{-\alpha})$ may
be used along with explicit computation of the $F(P,r)$ in terms of
$\Null\bpns{I(P,r)}$ to compute the absolute index of a given extension.

\subsection{Bounded polyhomogeneous coefficients} \label{S:bndd_phg_coeffs}
To avoid overburdening the notation and conceptual development of the previous
sections, we have limited the discussion to smooth objects. However, for the
application in \cite{kottke2013dimension} we need to generalize the operators
under consideration by weakening the requirement that their coefficients
be smooth. Instead we consider operators with {\em bounded
polyhomogeneous} coefficients. The general spaces $\cAphg^E(X; V)$ of
polyhomogeneous sections of a vector bundle with index set $E$ are defined in
Appendix~\ref{S:phg}; briefly, this means $E \subset \bbC\times \bbN$ is a
countable discrete set satisfying some technical conditions and $\cAphg^E(X;V)$
consists of sections which are smooth in the interior and having asymptotic
expansions at $\pa X$ of the form $\sum_{(z,k) \in E} x^z(\log x)^k v_{z,k}$,
with $v_{z,k} \in C^\infty(\pa X; V).$

We define the {\em bounded polyhomogeneous} space 
\begin{equation}
	\cB(X;V) = \bigcup_{E \subset (0,\infty)\times \bbN \cup \set{0}\times \set{0}} \cAphg^E(X;V)
	\label{E:bndd_phg}
\end{equation}
to consist of sections with asymptotic expansions in non-negative real powers
of $x$ which are bounded up to $\pa X$; in particular these are smooth on the
interior of $X$ and continuous up to $\pa X$ with smooth restrictions to $\pa
X$, and $\cB(X) = \cB(X; \bbC)$ is an algebra with respect to pointwise
multiplication. With this space in mind, we now indicate how the preceding
sections may be generalized. 

The spaces $\cB\bV(X)$ and $\cB\scV(X)$ of bounded polyhomogeneous
b and scattering vector fields are defined by taking $V = \bT X$ or $V = \scT
X$ in \eqref{E:bndd_phg}. These act as derivations on $\cB(X)$ and satisfy the
obvious analogues of \eqref{E:bb_comm}--\eqref{E:scsc_comm}. The spaces
\[
	\cB\bDiff^\ast(X), \quad\text{and}\quad \cB\scDiff^\ast(X)
\]
of bounded polyhomogeneous b and scattering differential operators may be
defined respectively as the enveloping algebras of $\cB\bV(X)$ (resp.\
$\cB\scV(x)$) over $\cB(X)$. Elements of these spaces act as bounded operators
on $\cB(X)$.

Bounded polyhomogeneous connections of the various types (true, b, and
scattering) are defined by taking the global connection form in $\cB(P;
{}^{(\mathrm{b}/\phi)}T^\ast P\otimes \mathfrak{g})$ or, equivalently, taking
the local connection forms in $\cB(X; {}^{(\mathrm{b}/\mathrm{sc})}T^\ast X\otimes\mathfrak{g}).$
The considerations of \S\ref{S:bkg_cnxns} apply, replacing $C^\infty$ by $\cB$
throughout, and lead to the definition of the spaces $\cB\bDiff^\ast(X; V)$ and
$\cB\scDiff^\ast(X;V)$ of bounded polyhomogeneous differential operators on a
vector bundle. 

These operators admit bounded extensions to the Sobolev spaces $x^\alpha
\bH^k(X;V)$, $x^\beta \scH^k(X;V)$ and $x^\alpha \bscH^{k,l}(X;V)$ just as
their smooth counterparts do, and the results of Theorems \ref{T:sc_thm},
\ref{T:b_fredholm} and \ref{T:b_relindex} extend essentially verbatim to
appropriately elliptic elements of $\cB\scDiff^\ast(X;V\otimes \scOh)$ and
$\cB\bDiff^\ast(X;V\otimes\bOh)$, {\em except} that the expansion
\eqref{E:b_phg_expn} is generally more complicated. In this case of an elliptic
element $P \in \cB\bDiff^k(X;V\otimes \bOh)$, we will only make use of the
weaker result 
\[
\begin{gathered}
	u \in \Null(P_\alpha) \implies u \in \cAphg^*(X;V\otimes \bOh), \\ 
	u = \cO\big(x^z(\log x)^{\ord(z)}\big), \quad \re z = \min\set{\re s : s \in
	\bspec(P), \re s > \alpha}, 
\end{gathered}
\]
In other words, the nullspace of a Fredholm extension of $P$ consists of
polyhomogeneous sections with {\em some} index set, and we only keep track of
its leading order.

The reader interested only in smooth coefficients can freely ignore these
generalizations and replace $\cB$ by $C^\infty$ throughout.

\section{Fredholm extensions of Callias operators} \label{S:fred}

We now return to the situation described in the introduction. Namely, let $X$
be an odd dimensional manifold with boundary equipped with a scattering metric
\eqref{E:exact_scattering_metric} and make the following assumptions:
\begin{enumerate}
[{\normalfont (C1)}]
\item\label{I:C_one}
$D \in \cB\scDiff^1(X; V\otimes \scOh)$ is a scattering Dirac or Dirac-type
operator acting on half-density sections (see \S \ref{S:bkg_dens}) of a
Hermitian $\scCl(X)$ module $V\to X$. In the case of a Dirac-type operator, we assume
that it differs from a true self-adjoint Dirac operator by a real term of order $\cO(x).$
\item\label{I:C_two}
The connection $\nabla$ on $V$ used to define $D$ is the lift of a
b connection --- in particular it induces a connection on $V \rst_{\pa X}$.
(Recall that this is automatically satisfied in particular if $\nabla$ is a
Levi-Civita connection with respect to $g$.)
\item\label{I:C_three}
$\Phi \in \cB\bpns{X; \End(V\otimes\scOh)}$ is skew-Hermitian, 
and $\Phi \rst_{\pa X}$ commutes with Clifford multiplication and has constant
rank over $\pa X$.  In particular,
\begin{equation}
	V \rst_{\pa X} = V_0\oplus V_1, \quad V_0 = \Null(\Phi \rst_{\pa X})
	\label{E:V_infty_splitting}
\end{equation}
splits into a sum of bundles which are preserved by the Clifford action. 
\item\label{I:C_four}
With respect to an extension of the splitting \eqref{E:V_infty_splitting} to
neighborhood $U \supset \pa X$, $\nabla$ decomposes as an operator from $\cB(U;
V_0\oplus V_1)$ to $\cB\big(U; \scT^\ast X\otimes(V_0 \oplus V_1)\big)$ as
follows\footnote{Note that the decay is the only extra condition here; the
off-diagonal terms are automatically 0th order as differential operators, which
follows by considering the difference $\nabla - \nabla'$ where $\nabla'$ is a
connection diagonal with respect the splitting.}:
\begin{equation}
	\nabla = \begin{pmatrix} \nabla_{00} & \nabla_{10} \\ \nabla_{01} & \nabla_{11} \end{pmatrix},
	\quad \nabla_{10}, \nabla_{01} \in x^{1 + \e} \cB\big(U; \Hom(V_i; \scT^\ast X\otimes V_{i+1})\big),
	\label{E:nabla_decay}
\end{equation}
for some $\e > 0.$ This follows for
instance if $\nabla^{\End(V)} \Phi = \cO(x^{1+\e})$, where $\nabla^{\End(V)}$
denotes the extension of $\nabla$ to the bundle $\End(V).$
\item\label{I:C_five}
Likewise, with respect to the extension of the splitting, 
\begin{equation}
	\Phi = \begin{pmatrix} \Phi_{00} & \Phi_{10} \\ \Phi_{01} & \Phi_{11} \end{pmatrix}, \quad \Phi_i = \cO(x^{1+\e'}), \ i \in \set{00,01,10}
	\label{E:phi_decay}
\end{equation}
for some $\e' > 0.$ 
\end{enumerate}
Given such data, we set
\[
	P = D + \Phi \in \cB\scDiff^1(X; V\otimes \scOh)
\]
and refer to it as a {\em (generalized) Callias type operator}.
\begin{rmk}
While the decay conditions (C\ref{I:C_four}) and (C\ref{I:C_five}) may at first
glance seem contrived, they are necessary for the results below. More
importantly, they are satisfied in the case of monopoles in
\cite{kottke2013dimension}. Indeed, in that setting the extension may be taken
such that $\Phi_{00}$, $\Phi_{10}$ and $\Phi_{01}$ vanish identically on $U$,
and finite energy/action of the configuration data implies that $\nabla^{\End(V)}\Phi
= \cO(x^{1 + \e})$ for some $\e > 0.$
\end{rmk}

In discussing Fredholm extensions of $P$, we first consider the two extreme
cases in which the rank of $\Phi \rst_{\pa X}$ is maximal or zero.

\subsection{The case of full rank} \label{S:fred_fullrank}

When $\Phi \rst_{\pa X}$ is invertible the following version of the classical
Callias index theorem is a consequence of Theorem~\ref{T:sc_thm} and the main
result proved in \cite{kottke2011index}. As discussed in the introduction, we
denote by $\pa^+_+ \in \Diff^1(\pa X; V^+_+\otimes \scOh, V^-_+ \otimes \scOh)$
the graded Dirac operator induced on $\pa X$ from $D$ acting on $V_+ \rst_{\pa
X} = V^+_+\oplus V^-_+$, where $V_+$ denotes the span of the positive imaginary
eigenvectors of $\Phi  \rst_{\pa X}$ and the splitting $V^+_+\oplus V^-_+$ is
with respect to the $\pm i$ eigenspaces of $i \cl (x^2 \pa_x)$.

\begin{thm}
Under the assumptions above on $D$ and $\Phi$ and with the additional
assumption that $\Phi \rst_{\pa X}$ is nondegenerate, $P$ is fully elliptic
with Fredholm extensions
\[
	P : x^\beta \bscH^{k,l}(X; V\otimes \scOh) \to x^\beta \bscH^{k,l-1}(X; V\otimes \scOh)
\]
for all $\beta,\ k,l$, and
\[
	\ind(P) = \ind(\pa^+_+).
\]
Elements
of $\Null(P)$ are smooth and rapidly vanishing at $\pa X$:
\[
	\Null(P) \subset x^\infty C^\infty(X; V\otimes \scOh).
\]
\label{T:calliasthm}
\end{thm}

\subsection{The case of zero rank} \label{S:fred_zero}
In the case that $\Phi \rst_{\pa X} \equiv 0$, $P$ fails to be fully elliptic
as a scattering operator and Theorem~\ref{T:calliasthm} does not hold. However
since $\Phi = \cO(x^{1+\e'})$ by \eqref{E:phi_decay} in this case, $P$ may be instead
considered as a weighted b differential operator using Lemma~\ref{L:lift_cnxn}.

There is a choice involved in the way $x$ is factored out; indeed $x^{-\gamma}
Px^{\gamma - 1}$ for $0 \leq \gamma \leq 1$ are essentially equivalent as
b-operators, though their indicial roots are shifted relative to one another.  We
use the convention $\gamma = 1/2$, which has the virtue of preserving formal
self-adjointness of the Dirac operator on the unweighted $L^2$ space, and under
which convention the b-spectrum of a Dirac operator is symmetric about $0$.  In
addition, as the Fredholm results for b-operators are most naturally stated
using b half densities, we also conjugate $P$ by $x^{n/2}$ and make use of the
equivalence \eqref{E:bsc_L2} between scattering half densities and b half
densities.

Thus as a notational convention, for any $Q \in \cB\scDiff^1(X; V\otimes \scOh)$ we define
\begin{equation}
	\wt Q := x^{-(n+1)/2} Q x^{(n-1)/2} \in \cB\bDiff^1(X; V\otimes \bOh),
	\label{E:fred_conj_conv}
\end{equation}
and note that mapping properties (boundedness, Fredholm, etc.) of 
\[
	Q : x^{\alpha - 1/2} \bH^k(X; V\otimes \scOh) \to x^{\alpha + 1/2} \bH^{k-1}(X; V\otimes \scOh)
\]
are equivalent to the corresponding mapping properties of
\[
	\wt Q : x^{\alpha} \bH^k(X; V\otimes \bOh) \to x^{\alpha} \bH^{k-1}(X; V\otimes \bOh).
\]
Moreover $\wt Q$ is formally self-adjoint with respect to $L^2(X; V\otimes
\bOh)$ if and only if $Q$ is formally self-adjoint with respect to $L^2(X;
V\otimes \scOh)$.

Given the assumption that $\Phi = \cO(x^{1+\e'}),$ it follows that $\wt P = \wt
D + \wt \Phi$ where $\wt \Phi = x^{-1}\Phi = \cO(x^{\epsilon'}).$ In particular
$\wt P$ is a compact perturbation of $\wt D$. From standard results of the b
calculus (summarized in \S \ref{S:bkg_b_calc}) we conclude:

\begin{prop}
$P$ admits a Fredholm extension
\begin{equation}
	P : x^{\alpha - 1/2} \bH^k(X; V\otimes \scOh) \to x^{\alpha + 1/2} \bH^{k-1}(X; V\otimes \scOh)
	\label{E:fred_rank0_noext}
\end{equation}
for all $\alpha \notin \bspec(\wt D)$ and all $k,$ and $\Null(P)$ consists of
polyhomogeneous sections with expansions:
\[
\begin{gathered}
	\Null(P) \subset \cAphg^\ast(X; V\otimes \scOh), \\
	\Null(P) \ni u = \cO\big(x^{r + (n-1)/2}(\log x)^{\ord(r)}\big), \quad r = \min\set{s \in \bspec(\wt D) : s > \alpha}.
\end{gathered}
\]
If $P$ has smooth coefficients, then in fact $\Null(P) \subset \cAphg^{\wh
E^+(\alpha) + (n-1)/2}(X; V\otimes \scOh)$ where $E^+(\alpha) = \set{(r + n, k)
\in \R\times \N : (r,k) \in \bspec(\wt D),\ n \in \N, r > \alpha}$ and  $\wh
E^+(\alpha) = \ol\bigcup_k \big(E^+(\alpha) + k\big)$.
The index $\ind(P,\alpha)$ of the extension \eqref{E:fred_rank0_noext}
satisfies
\begin{equation}
\begin{aligned}
	\ind(P,-\alpha) &= -\ind(P,\alpha), \quad \text{if $\wt D$ is self-adjoint, and} \\
	\ind(P,\alpha_0 - \epsilon) &= \ind(P,\alpha_0 + \epsilon) + \dim F(\wt D,\alpha_0)
\end{aligned}
	\label{E:ind_rank0_props}
\end{equation}
for $\alpha_0 \in \bspec(\wt D)$ and $\epsilon > 0$ sufficiently small so that
$\bspec(\wt D) \cap [\alpha_0-\epsilon,\alpha_0+\epsilon] = \set{\alpha_0}.$
\label{P:fred_rank0_noext}
\end{prop}
\begin{proof}
The properties \eqref{E:ind_rank0_props} follow from the fact that $\wt P$ has
the same index as $\wt D$ as an operator $x^\alpha \bH^k \to x^\alpha
\bH^{k-1}$ which satisfies analogous properties by
Theorems~\ref{T:b_fredholm} and \ref{T:b_relindex}.

The characterization of $\Null(P)$ comes from the fact that $\Null(P) =
x^{(n-1)/2}\Null(\wt P)$ which involves a multiplication by $x^{-1/2}$ along
with the natural identification $\bOh(X) = x^{n/2}\scOh(X)$.
\end{proof}

\subsection{General constant rank nullspace} \label{S:fred_mixed}
We now consider the general case described at the beginning of this section,
wherein $V \rst_{\pa X} = V_0\oplus V_1$ according to $V_0 = \Null(\Phi)$,
$\Phi \rst_{V_1} \neq 0,$ and the splitting extends to a neighborhood $U$ of
$\pa X$ for which \eqref{E:nabla_decay} and \eqref{E:phi_decay} hold.

To obtain a Fredholm result we are forced to measure regularity near $\pa X$
differently according to the splitting of $V.$ This is accomplished by defining
the following families of hybrid Sobolev spaces.

\begin{defn}
Let $\Pi_0 \in C^\infty(U; \End(V_0\oplus V_1))$ be the projection onto the $V_0$
subbundle, and denote by $\Pi_1 = \pns{\id - \Pi_0}$ the other projection.  Let
$\chi \in C^\infty(X; [0,1])$ be a cutoff supported on $U$ and such that
$\chi \equiv 1$ on an open neighborhood of $\pa X$ and let
\begin{equation}
	\cH^{\alpha,\beta,k,l}(X; V \otimes \scOh)
	\label{E:hybrid_sobolev}
\end{equation}
consist of those $u \in C^{-\infty}(X; V\otimes \scOh)$ such that
\[
\begin{aligned}
	\Pi_0 \chi u &\in x^\alpha \bH^{k+l}(U; V_0\otimes \scOh), \\
	\Pi_1 \chi u &\in x^\beta \bscH^{k,l}(U; V_1\otimes \scOh), \\
	(1 - \chi) u &\in H_c^{k+l}(X; V\otimes \scOh).
\end{aligned}
\]
\end{defn}

It follows from the fact that the $\bH^{k+l}$ and $\bscH^{k,l}$ norms are
equivalent on sections supported away from $\pa X$ that
$\cH^{\alpha,\beta,k,l}(X; V\otimes \scOh)$ is a well-defined, complete Hilbert
space with respect to the norm
\[
	\norm{u}_{\cH^{\alpha,\beta,k,l}}^2 = \norm{\Pi_0 \chi u}_{x^\alpha\bH^{k+l}}^2 + \norm{\Pi_1 \chi u}_{x^\beta \scbH^{k,l}}^2 + \norm{(1 - \chi)u}_{H_c^{k+l}}^2
\]
and that it is independent of the choice of $\chi$. The corresponding inner
product is obtained by polarization.

It follows from the assumptions (C\ref{I:C_one})--(C\ref{I:C_five}) that $P$ has the form 
\begin{equation}
	P = D + \Phi = \begin{pmatrix} D_0 & L_{01}\\L_{10}& D_1\end{pmatrix} + \Phi,
	\quad L_i \in x^{1+\e}\cB\big(X; \End(V\otimes \scOh)\big),
	\label{E:fred_P_splitting}
\end{equation}
near $\pa X$, for a fixed $\e > 0.$
 
The term $D_0$ in \eqref{E:fred_P_splitting} is only defined on the
neighborhood $U$ of $\pa X$; nevertheless, we may consider $\wt D_0 =
x^{-(n+1)/2}D_0 x^{(n-1)/2}$ as in \S \ref{S:fred_zero} and refer to the
indicial family $I(\wt D_0,\lambda) \in \Diff(\pa X; V_0\otimes \Oh)$ and the
indicial root spectrum $\bspec(\wt D_0)$ since these only depend on the formal expansion
$I(\wt D_0)$ of $\wt D_0$ at $\pa X.$

\begin{thm}
\mbox{}
\begin{enumerate}
[{\normalfont (a)}]
\item $P$ is bounded as an operator
\[
	P : \cH^{\alpha-1/2,\beta,k,l}(X; V\otimes \scOh) \to \cH^{\alpha+1/2,\beta,k,l-1}(X; V\otimes \scOh)
\]
for any $\alpha, \beta \in \R$, $k \in \N$ such that $1 \leq l \leq (\beta - \alpha)  + 3/2 \leq 2.$
\item
If $\alpha \notin \bspec(\wt D_0)$ then the extension
\begin{equation}
	P : \cH^{\alpha-1/2,\alpha+1/2,k,1}(X; V\otimes \scOh) \to \cH^{\alpha+1/2,\alpha+1/2,k,0}(X; V\otimes \scOh)
	\label{E:fredholm_hybrid_general}
\end{equation}
is Fredholm.
\item
The nullspace of the extension \eqref{E:fredholm_hybrid_general} consists of
polyhomogeneous sections with leading order
\[
	u = u_0\oplus u_1, \quad u_0 = \cO\big(x^{r+(n-1)/2}(\log x)^{\ord(r)}\big), 
	\quad u_1 = \cO\big(x^{r+1+(n-1)/2}(\log x)^{\ord(r)}\big)
\]
where $r = \min\set{s \in \bspec(\wt D_0) : s > \alpha}.$ If $P$ has
smooth coefficients, then 
\[
	u_0 \in \cAphg^{\wh E^+(\alpha) + (n-1)/2}, \quad u_1 \in \cAphg^{\wh E^+(\alpha) + 1 + (n-1)/2},
\]
where $\wh E^+(\alpha)$ is defined as in Proposition~\ref{P:fred_rank0_noext}
with respect to $\bspec(\wt D_0).$
\end{enumerate}
\label{T:fredholm_hybrid_norealroot}
\end{thm}
\begin{proof}
For (a), boundedness along the diagonal components is clear from the discussion
in the previous two sections. For the off-diagonal components, we observe that,
near $\pa X$,
\[
\begin{aligned}
	L_{10} &: x^{\alpha - 1/2} \bH^{k+l} \to x^{\alpha + 1/2} \bH^{k+l} \subseteq x^{\beta} \bscH^{k,l-1} \\
	L_{01} &: x^{\beta} \bscH^{k,l} \to x^{\beta + 1} \bscH^{k,l} \subseteq x^{\alpha+1/2} \bH^{k + l - 1}
\end{aligned}
\]
provided $\alpha + 1/2 \geq \beta$ in the first case and, by
Proposition~\ref{P:compact_inclusion}.(c), $\beta + 1 \geq \alpha + 1/2 + (l -
1)$ in the second case. Combining these gives $1 \leq l \leq \beta - \alpha + 3/2 \leq 2$.

For (b),
we proceed to construct a parametrix $Q$ decomposing as a direct sum of terms
near $\pa X^2$ coming from the scattering and b calculi, respectively. Thus let
$Q \in C^{-\infty}\bpns{X^2; \End(V\otimes \scOh)}$ be a distribution supported
near the diagonal and conormal to it on the interior of $X^2$ whose restriction
to $U\times U$ has the form
\[
\begin{gathered}
	Q\rst_{U^2} = \begin{pmatrix} Q_0 & 0 \\ 0 & Q_1 \end{pmatrix}, 
	  \quad Q_0 = x^{(n-1)/2} \wt Q_0 {x'}^{-(n+1)/2}, 
	\\ \wt Q_0 \in \bP^{-1,\cE}(U; V_0\otimes \bOh),
	  \quad Q_1 \in \scP^{-1,0}(U; V_1\otimes \scOh)
\end{gathered}
\]
for $\wt Q_0$ and $Q_1$ yet to be determined, where $x$ and $x'$ denote the
lifts of $x$ from the left and right respectively. (Strictly speaking, we
should say $\wt Q_0$ is a pushforward to $U^2$ by the blow-down map
$U_\mathrm{b}^2 \to U$, and likewise $Q_1$ is a pushforward by the blow-down
map from $U_{\mathrm{sc}}^2$, but for clarity of notation we will suppress
this.)

Let $\sigma_{-1}(Q_1) = \sigma_{1}(D_1)^{-1}$, and $\sigma_{-1}(\wt Q_0) =
\sigma_1(\wt D_0)^{-1}$, which are each consistent with the further condition that
$\sigma_{-1}(Q) = \sigma_1(P)^{-1}$ on the interior of $X^2.$ In addition, let
$\ssym(Q_1) = \ssym(D_1 + \Phi_{11})^{-1}$ and $I(\wt Q_0) = \cM^{-1}_{\re \lambda
= \alpha} I(\wt D_0,\lambda)^{-1}$. 

By lifting $P$ to $X^2$ from the left or right and composing with $Q$, it
follows that that $R_L := I - QP$ and $R_R := I - PQ$ are distributions conormal
of order $-1$ with respect to the interior diagonal, where
\[
	R_L \rst_{U^2}=  \begin{pmatrix} R_L^0 & -Q_0L_{01} \\ - Q_1L_{10} & R_L^1 \end{pmatrix}, 
	\quad R_R \rst_{U^2} =  \begin{pmatrix} R_R^0 & -L_{01} Q_1 \\ -L_{10} Q_0 & R_R^1 \end{pmatrix}.
\]
Here $R^1_{R} = I - D_1Q_1$ and $R^1_L = I - Q_1D_1$ lie in $\scP^{-1,\gamma}$,
and $R^0_{R/L} = x^{(n\pm1)/2} \wt R^0_{R/L} {x'}^{-(n\pm 1)/2}$ where
\[
	\wt R^0_R = I - \wt D_0\wt Q_0, \ \wt R^1_L = I - \wt Q_0\wt D_0 \in \bP^{-1,(\gamma,E^+(\alpha),E^-(\alpha))}(U; V_0\otimes \bOh),
\]
$E^\pm(\alpha) = \set{(\pm z + n,k) : (z,k) \in \bspec(\wt D_0),\ \re z \gele
\alpha,\ n \in \bbN_0}$, and $\gamma > 0$ is the subleading order of the
asymptotic expansion of $P$ at $\pa X$ (in the smooth case $\gamma = 1.$)

Setting $\beta = \alpha+1/2$, we claim that $R_L$ and $R_R$ are compact as operators on
$\cH^{\alpha-1/2,\alpha+1/2,k,1}(X;V\otimes \scOh)$ and
$\cH^{\alpha+1/2,\alpha+1/2,k,0}(X; V\otimes \scOh)$, respectively. We restrict
attention to the neighborhood of the boundary, since the smoothing properties
of $R_{R/L}$ in the interior are clear. Indeed,
\[
\begin{aligned}
	R_L^0 &: x^{\alpha-1/2} \bH^{k+1} \to x^{\alpha+1/2}\bH^{k+2} \subset x^{\alpha-1/2}\bH^{k+1}, \\
	R_L^1 &: x^{\alpha+1/2} \bscH^{k,1} \to x^{\alpha+3/2}\bscH^{k,2} \subset x^{\alpha+1/2}\bscH^{k,1}, \\
	Q_1L_{10} &: x^{\alpha-1/2} \bH^{k+1} \to x^{\alpha+1/2+\e}\bscH^{k+1,1} \subset x^{\alpha+1/2}\bscH^{k,1}, \\
	Q_0L_{01} &: x^{\alpha+1/2} \bscH^{k,1} \to x^{\alpha+1/2+\e}\bscH^{k+1,1} \subset x^{\alpha-1/2}\bH^{k+1},
\end{aligned}
\]
where all inclusions are compact by Proposition~\ref{P:compact_inclusion}, and
then this implies the compactness result
for $R_L$. Similarly, compactness of the inclusions in
\[
\begin{aligned}
	R_R^0 &: x^{\alpha+1/2} \bH^{k} \to x^{\alpha+3/2}\bH^{k+1} \subset x^{\alpha+1/2}\bH^k, \\
	R_R^1 &: x^{\alpha+1/2} \bH^k \to x^{\alpha+3/2}\bscH^{k,1} \subset x^{\alpha+1/2}\bH^k, \\
	L_{10}Q_0 &: x^{\alpha+1/2} \bH^k \to x^{\alpha+1/2+\e}\bH^{k+1} \subset x^{\alpha+1/2}\bH^k, \\
	L_{01}Q_1 &: x^{\alpha+1/2} \bH^{k} \to x^{\alpha+3/2+\e}\bscH^{k,1} \subset x^{\alpha+1/2}\bH^k
\end{aligned}
\]
implies the result for $R_R.$ 

To determine the regularity of the nullspace, we need to improve the parametrix
construction further. Smoothness in the interior follows from standard elliptic
regularity, so we focus again on the behavior near the boundary. Proceeding
with standard methods in the b- and sc- calculi, we may improve $Q_0$ and $Q_1$
so that the error terms are smoothing; in particular
\[
\begin{gathered}
	R^1_L : x^{\beta}H^{-\infty} \to x^\infty C^\infty, \\
	\wt R^0_L : x^{\alpha}H^{-\infty} \to \cAphg^{E}
	\iff R^0_L : x^{\alpha-1/2}H^{-\infty} \to \cAphg^{E + (n-1)/2}
\end{gathered}
\]
locally near $\pa X,$ where $E$ is some index set with leading order $\min\set{s
\in \bspec(\wt D_0) : s > \alpha}$. In the smooth case $E = \wh E^+(\alpha)$.
We set $F = \bigcup_{k \in \bbN} (E + (n-1)/2 + k\e)$, where $\e$ is as in
\eqref{E:fred_P_splitting} (in the smooth case, $\e$ is a positive integer, so
$F \equiv \wh E^+(\alpha) + (n-1)/2$). We will make use of the fact that $\cAphg^{E+
(n-1)/2} \subset \cAphg^F$, and that $\cAphg^{F + \e} \subset \cAphg^F.$

Thus, supposing that $u \in \Null(P) \cap \cH^{\alpha-1/2,\alpha+1/2,0,1}(X;
V\otimes \scOh)$, (we take $k = 0$ for simplicity since the ultimate result
will be manifestly independent of $k$), so on the neighborhood $U \supset \pa
X$,
\[
	u = u_0 \oplus u_1, \quad u_0 \in x^{\alpha-1/2}\bH^1(U; V_0\otimes \scOh),\ u_1 \in x^{\alpha+1/2}\scH^1(U; V_1\otimes \scOh),
\]
we make an inductive argument.

First, it follows from $R_Lu = u - QPu = u$ that
\[
\begin{aligned}
	u_0 &= R^0_L u_0 - Q_0L_{01}u_1 \in \cAphg^{F} + x^{\alpha+1/2+\e} \bscH^{1,1} \\
	u_1 &= R^1_L u_1 - Q_1L_{10}u_0 \in x^{\infty}C^\infty + x^{\alpha+1/2+\e} \bscH^{1,1}.
\end{aligned}
\]
The rightmost term of $u_0$ is included the space $x^{\alpha-1/2}\bscH^{1,1}$
with some decay; thus we may write
\[
\begin{aligned}
	u_0 &= \bar u^1_0 + x u^1_0, & \bar u^1_0 &\in \cAphg^{F},\ u^1_0 \in x^{\alpha-1/2}\bscH^{1,1} \\
	u_1 &= u^1_1, & u^1_1 &\in x^{\alpha+1/2}\bscH^{1,1}.
\end{aligned}
\]
Applying $u = R_L u$ again gives
$u_0 = \bar u^{2}_0 + x u^{2}_0$ and $u_1 = \bar u^{2}_1 +
x u^{2}_1$, where 
\[
\begin{aligned}
	\bar u^{2}_0 &= R^0_L(u_0) \in \cAphg^{F}\\
	u^{2}_0 &= - x^{-1}Q_0L_{01}(u^{1}_1) \in x^{\alpha-1/2}\bscH^{2,1}\\
	\bar u^{2}_1 &= R^1_L(u_1) - Q_1L_{10}(\bar u^{2}_0) \in \cAphg^{F + 1 + \e} \subset \cAphg^{F + 1}\\
	u^{2}_1 &= - x^{-1} Q_1L_{10}(x u^{1}_0) \in x^{\alpha+1/2}\bscH^{1,2}.
\end{aligned}
\]
Proceeding inductively, we may assume that
$u_0 = \bar u^{2k}_0 + x^{k} u^{2k}_0$ and $u_1 = \bar u^{2k}_1 + x^{k}u^{2k}_1$, where 
\[
\begin{aligned}
	\bar u^{2k}_0 &\in \cAphg^{F},  &
	u^{2k}_0 &\in x^{\alpha-1/2}\bscH^{k+1,k},  \\
	\bar u^{2k}_1 &\in \cAphg^{F + 1},  &
	u^{2k}_1 &\in x^{\alpha+1/2}\bscH^{k,k+1}.
\end{aligned}
\]
Applying $u = R_L u$ twice gives $u_0 = \bar u^{2k+1}_0 + x^{k+1}u^{2k+1}_0$ and $u_1 = \bar
u^{2k+1}_1 + x^{k}u^{2k+1}_1$ where
\[
\begin{aligned}
	\bar u^{2k+1}_0 &= R^0_L(u_0) - Q_0L_{01}(\bar u^{2k}_1) \in \cAphg^{F}\\
	u^{2k+1}_0 &= - x^{-(k+1)}Q_0L_{01}(x^{k}u^{2k}_1) \in x^{\alpha-1/2}\bscH^{k+1,k+1}\\
	\bar u^{2k+1}_1 &= R^1_L(u_1) - Q_1L_{10}(\bar u^{2k}_0) \in \cAphg^{F + 1}\\
	u^{2k+1}_1 &= - x^{-k} Q_1L_{10}(x^{k}u^{2k}_0) \in x^{\alpha+1/2}\bscH^{k+1,k+1},
\end{aligned}
\]
and then $u_0 = \bar u^{2(k+1)}_0 + x^{k+1}u^{2(k+1)}_0$ and $u_1 = \bar u^{2(k+1)}_1
+ x^{k+1}u^{2(k+1)}_1$ where
\[
\begin{aligned}
	\bar u^{2(k+1)}_0 &= R^0_L(u_0) - Q_0L_{01}(\bar u^{2k+1}_1) \in \cAphg^{F}\\
	u^{2(k+1)}_0 &= -x^{-(k+1)}Q_0L_{01}(x^{k}u^{2k+1}_1) \in x^{\alpha-1/2}\bscH^{k+2,k+1} \\
	\bar u^{2(k+1)}_1 &= R^1_L(u_1) - Q_1L_{10}(\bar u^{2k+1}_0) \in \cAphg^{F + 1}\\
	u^{2(k+1)}_1 &= -x^{-(k+1)} Q_1L_{10}(x^{k+1}u^{2k+1}_0) \in x^{\alpha+1/2}\bscH^{k+1,k+2},
\end{aligned}
\]
completing the induction. It follows that
\[
	u_0 \in \cAphg^{F}(U; V_0\otimes \scOh) + x^\infty \bscH^{\infty,\infty}(U; V_0\otimes \scOh) \equiv \cAphg^{F}(U; V_0\otimes \scOh)
\]
and similarly $u_1 \in \cAphg^{F+1}(U; V_1\otimes \scOh)$.
\end{proof}

\section{The index theorem} \label{S:index}
To compute the index of the Fredholm extensions of the last section, we employ
the following strategy. Consider the family
\[
	\Pfam := P - i\chi \tau\oplus 0 = D + \Phi - \chi\begin{pmatrix} i \tau & 0 \\ 0 & 0 \end{pmatrix}, \quad 0 \leq \tau < 1,
\]
with respect to the splitting $V \rst_{U} = V_0\oplus V_1$ and where $\chi \in
C^\infty_c(U)$ with $\chi \equiv 1$ at $\pa X$ as before.  For $\tau > 0$,
$\Pfam$ is a classical Callias type operator, as considered in \S
\ref{S:fred_fullrank}, with Fredholm extensions on scattering Sobolev spaces
and index 
\[
	\ind(\Pfam(\tau)) = \ind(\pa^+_+), \quad \tau > 0.
\]
At $\tau = 0$ we would like to consider the Fredholm extension
\eqref{E:fredholm_hybrid_general}, however this cannot simply be regarded as a
1-parameter family of Fredholm operators since the domains are discontinuous at
$\tau = 0.$ 

We therefore analyze $\Pfam$ in a framework designed to interpolate between
the elliptic scattering behavior of $\Pfam$  at $\tau > 0$ and the elliptic b
behavior at $\tau = 0.$ This framework is the {\em b-sc transition calculus} of
pseudodifferential operators. The basic space considered is the transition
single space
\[
	\tX := [X\times I; \pa X \times 0],
\]
(see Figure \ref{F:singlespace}) which is the radial blow-up of $X \times I$
(here $I = [0,1)$ denotes the interval in which the parameter $\tau$ varies) at
the corner $\pa X \times 0$. This space $\tX$ fibers over the parameter
interval $I$; for $\tau > 0$ the fibers are just $X$, but at $\tau = 0$ the
fiber is the union of two boundary faces denoted by $\fzf$, the `zero face'
which is itself diffeomorphic to $X$, and $\ftf$, the `transition face' which
is diffeomorphic to $\pa X \times [0,1]$. We view the latter as the
compactification of the space $\pa X \times \bbR$ with two infinite ends, where
$\pa X \times 0$ is the end joined to the boundary of $\fzf$.

Lifting an operator family such as $\Pfam$ to this space, we obtain not one but
{\em two} operators in the $\tau = 0$ limit: the restriction of $\Pfam$ to the
zero face, which is denoted by $N_\fzf(\Pfam)$, and the restriction (up to a
rescaling described below) of $\Pfam$ to the transition face, which is denoted
by $\Nt(\Pfam)$. The latter is an operator on $\pa X \times[0,1]$, which is of
elliptic scattering type at the $\pa X \times 1$ end, and elliptic b type at
the $\pa X \times 0$ end, and can be viewed as smoothly interpolating between
the elliptic behavior of $\Pfam(\tau)$ for $\tau > 0$ and $\tau = 0.$

The next step is to construct a parametrix for the family $\Pfam$ in the
calculus; restricted to $\tau > 0$, this furnishes a Fredholm parametrix for
$\Pfam(\tau)$ as a scattering operator, and restricted to $\fzf$ it furnishes
a Fredholm parametrix for $N_\fzf(\Pfam)$ as a b operator, for whichever choice
of weighted b Sobolev spaces is desired. Such a choice then dictates a certain
Fredholm extension of $\Nt(\Pfam)$, for which the restriction of the parametrix
to $\ftf$ provides an approximate inverse. With a sufficiently good parametrix
construction (so that the error terms are trace-class), one can then deduce the
index formula
\begin{equation}
	\ind\big(\Pfam(\tau)\big) = \ind\big(N_\fzf(\Pfam)\big) + \ind\big(\Nt(\Pfam)\big), \quad \forall\, \tau > 0,
	\label{E:index_family}
\end{equation}
relating the index of the desired Fredholm extension of $\Pfam(0)$, here
represented by $N_\fzf(\Pfam)$, to the index of $\Pfam(\tau)$. The difference
is given by the index of the model operator $\Nt(\Pfam)$ which may be computed
using the relative index theorem \ref{T:b_relindex}.

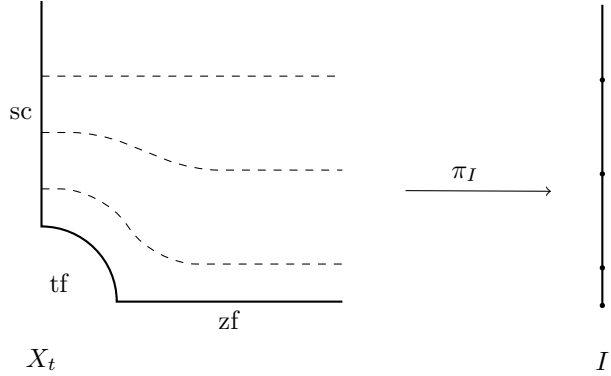
\begin{figure}[tb]
\begin{center}
\begin{tikzpicture}
\matrix (m) [matrix of nodes, column sep=3cm]
{
\begin{tikzpicture}
	\path 	(0,4) coordinate (a)
		(0,1) coordinate (b)	
		(1,0) coordinate (c)	
		(4,0) coordinate (d);
	\draw [thick]
	 	(a) -- (b)
		arc (90:0:1)
		-- (d);
	\draw[dashed] (0,3) -- (4,3);
	\draw[dashed, rounded corners=10pt] (0,2.25) -- (0.75,2.25) -- (2,1.75) -- (4,1.75);
	\draw[dashed, rounded corners=16pt] (0,1.5) -- (0.8,1.5) -- (1.5,0.5) -- (4,0.5);
	\node [left] at ($(a)!0.5!(b)$) {$\fsc$};
	\node [below left] at (0.5,0.5) {$\ftf$};
	\node [below] at ($(c)!0.5!(d)$) {$\fzf$};
	\node [below] at (0,-0.5) {$\tX$};
\end{tikzpicture} & 
\begin{tikzpicture}
	\draw[thick] (0,4) -- (0,0);
	\fill (0,3) circle (1pt);
	\fill (0,1.75) circle (1pt);
	\fill (0,0.5) circle (1pt);
	\fill (0,0) circle (1pt);
	\node [below] at (0,-0.5) {$I$};
\end{tikzpicture}\\};
\path (m-1-1) edge[->,>=to, shorten >=10pt, shorten <=20pt] node[auto] {$\pi_I$} (m-1-2);
\end{tikzpicture}
\caption{The single space $\tX$ and its boundary faces, along with some level sets of $\pi_I$.}
\label{F:singlespace}
\end{center}
\end{figure}

\subsection{Transition calculus} \label{S:fredholm_rank0}
We will give a rapid account of the transition calculus, deferring proofs and
technical discussion to Appendix~\ref{S:transition}, where the calculus is
developed in detail.

The {\em single space} $\tX := [X\times I; \pa X \times 0]$ has already been
mentioned.  There are three boundary faces of $\tX$ (see Figure
\ref{F:singlespace}) which are referred to as the `scattering face' $\fsc$, the
`transition face' $\ftf$ and the `zero face' $\fzf$. They are diffeomorphic to
$\pa X \times I$, $\pa X \times \bIsc$ (this notation will be explained
shortly) and $X$, respectively.

The projection $\pi_I : X\times I \to I$ lifts to a b-fibration (a type of
generalized fibration well suited to manifolds with corners --- see Appendix
\ref{S:phg} for a definition) $\pi_I : \tX \to I$ whose level sets
$\pi_I^{-1}(\tau)$, $\tau > 0$ are diffeomorphic to $X$, and whose level set
$\pi_I^{-1}(0)$ consists of the union of boundary faces $\fzf$ and $\ftf$.

The vector fields and by extension, differential operators, considered (see \S
\ref{S:calculus_diffl}) are those which are tangent to the fibers of $\pi_I$
and which restrict to scattering vector fields on $\pi_I^{-1}(\tau)$, $\tau >
0$, restrict to b vector fields on $\fzf$, and on $\ftf$ restrict to vector
fields which are of b-type near $\ftf \cap \fzf$ and scattering type near $\ftf
\cap \fsc$ (hence the notation $\ftf \cong \pa X \times \bIsc$).

It follows that the indicial families of the restrictions of such a vector
field $V$ to $\ftf$ and $\fzf$ are related by $I(N_{\ftf}(V), \lambda) =
I(N_{\fzf}(V),-\lambda)$, where $N_\ast(V)$ denotes the restriction of $V$ to a
boundary face (see Proposition~\ref{P:trans_diffl_indfam}).

\begin{rmk}
This relation between the indicial families is what dictates the Fredholm
extension of the model operator $\Nt(\Pfam)$ in terms of the chosen Fredholm
extension of $N_\fzf(\Pfam)$, as alluded to above.
In fact, $\tX$ may be alternatively thought of as a `surgery' or `gluing' space in which a
scattering end is glued onto the b end of a manifold with boundary by attaching
the cylindrical end $\pa X \times \bIsc,$ and the index theorem below may be
seen as a gluing formula for the index. The indicial family relation 
is then interpreted as a matching condition across the gluing hypersurface.
\end{rmk}

It is illustrative to consider the lift of scattering vector fields from $X$ to
$\tX$. If $(x,y)$ denote coordinates on $X$ as usual, with $y$ local
coordinates on $\pa X$ and $x$ boundary defining, then the scattering vector
fields are spanned by $x^2 \pa_x$ and $x\pa_{y}.$ There are corresponding local
coordinates $(x,y,\tau)$ on $X \times I$, and we obtain natural coordinates
\[
\begin{aligned}
	(\xi, y, \tau) &:= (x/\tau, y, \tau) \quad \text{near $\ftf \cap \fsc$, and} \\
	(x,y,\eta) &:= (x,y,\tau/x) \quad \text{near $\ftf \cap \fzf$}
\end{aligned}
\]
on the blow-up $\tX.$ In the first set of coordinates, $\tau$ is a local
boundary defining function for $\ftf$ and $\xi$ is boundary defining for
$\fsc$, while in the second set $x$ is a local boundary defining function for
$\ftf$ and $\eta$ is boundary defining for $\fzf$. Lifting the scattering
vector fields to $\tX$, we obtain 
\begin{equation}
\begin{aligned}
	\set{x^2\pa_x,x\pa_y} &\mapsto \set{\tau \xi^2 \pa_\xi, \tau \xi \pa_y}, &&\text{near $\ftf \cap \fsc$, and} \\
	\set{x^2\pa_x,x\pa_y} &\mapsto \set{x^2\pa_x - x\eta\pa_\eta, x\pa_y}, &&\text{near $\ftf \cap \fzf$.}
\end{aligned}
	\label{E:lift_sc_to_trans}
\end{equation}
Thus, after factoring out an overall vanishing factor at $\ftf$ (represented by
$\tau$ in the first case and $x$ in the second), one obtains scattering vector
fields $\set{\xi^2\pa_\xi,\xi\pa_y}$ on $\ftf$ near $\fsc$, b vector fields
$\set{-\eta \pa_\eta,\pa_y}$ on $\ftf$ near $\fzf$, and b vector fields
$\set{x\pa_x, \pa_y}$ on $\fzf.$ Observe also that the vector field $x^2\pa_x$
lifts to $x(x \pa_x)$ at $\fzf$ and to $x(-\eta \pa_\eta)$ at $\ftf$, which is
consistent with the indicial relation $I\big(\Nt(V),\lambda\big) =
I\big(N_\fzf(V),-\lambda\big)$.

\medskip

In any case, restricting differential operators to various submanifolds results in the
following homomorphisms of differential operator algebras:
\begin{equation}
\begin{aligned}
	N_\tau  &: \cB\tDiff^k(\tX; V\otimes \bOh) \to \cB\scDiff^k(X; V\otimes \bOh), \ \tau > 0 \\
	\Nt &: \cB\tDiff^k(\tX; V\otimes \bOh) \to \cB\bscDiff^k(\pa X \times \bIsc; V\otimes \bOh) \\
	N_\fzf &: \cB\tDiff^k(\tX; V\otimes \bOh) \to \cB\bDiff^k(X; V\otimes \bOh).
\end{aligned}
	\label{E:tdiffl_normalops}
\end{equation}
Here $N_\tau$ denotes restriction to the submanifold $\pi_I^{-1}(\tau)$, $\tau > 0.$
The symbols and indicial operators of these objects are compatible in the sense
that 
\begin{equation}
\begin{aligned}
	\ssym(\Nt(P)) &= \lim_{\tau \smallto 0} \ssym(N_\tau(P)), \\
	I(\Nt(P),\lambda) &= I(N_\fzf(P),-\lambda),
	\label{E:tdiffl_compat}
\end{aligned}
\end{equation}
(see Proposition~\ref{P:calc_normal_relns}).
We refer to \eqref{E:tdiffl_normalops} as the `normal operator' homomorphisms.
Restriction to the boundary face $\fsc$ gives fiberwise translation invariant
differential operators on the bundle $\pi_X^\ast \scT X \rst_{\fsc} \to \fsc
\cong \pa X \times I$, and taking the Fourier transform with respect to the
vector bundle structure gives a family of the scattering symbols which we
denote $\ssym(P)$ and which satisfy
\begin{equation}
\begin{aligned}
	\ssym(P)(\tau) &= \ssym(N_\tau(P)), \quad \tau > 0, \\
	\ssym(P)(0) &=  \ssym(P) \rst_{\fsc \cap \ftf} = \ssym(N_\ftf(P)),
	\label{E:tdiffl_compat2}
\end{aligned}
\end{equation}
(see Proposition~\ref{P:calc_normal_relns}).

Pseudodifferential operators $\tP^\ast(\tX; V\otimes \bOh)$ are defined in
terms of their Schwartz kernels on the {\em double space} $\tX^2$, a blown-up
version of $X^2 \times I$ (see \S \ref{S:calculus_spaces}), and there are
analogous normal operator homomorphisms at the pseudodifferential level (see
Proposition~\ref{P:tra:normal_op_comp}):
\[
\begin{aligned}
	N_\tau &: \tP^\ast(\tX; V\otimes \bOh) \to \scP^\ast(X; V\otimes \bOh), \ \tau > 0 \\
	\Nt &: \tP^\ast(\tX; V\otimes \bOh) \to \bscP^\ast(\pa X \times \bIsc; V\otimes \bOh) \\
	N_\fzf &: \tP^\ast(\tX; V\otimes \bOh) \to \bP^\ast(X; V\otimes \bOh) \\
	\ssym &: \tP^\ast(\tX; V\otimes \bOh) \to C^\infty\bpns{\scT^\ast_{\pa X} X\times I; \End(V\otimes \bOh)}
\end{aligned}
\]
These satisfy the same compatibility relations as in \eqref{E:tdiffl_compat}
and \eqref{E:tdiffl_compat2}.

Finally, for appropriately well-behaved operators there is a trace (see \S\ref{S:calculus_trace})
\begin{equation}
	\Tr : \tP^\ast(\tX; V\otimes \bOh) \to \cAphg^\ast(I)
	\label{E:families_trace}
\end{equation}
where the range space denotes functions which are smooth on the interior of $I$
with polyhomogeneous expansions --- in terms of $\tau^{z}(\log \tau)^k$, $z \in
\C$, $k \in \N$ --- at $\tau = 0$ (see Appendix~\ref{S:phg}). As a function of
$\tau$, this trace is continuous and
\begin{equation}
\begin{aligned}
	\Tr(P)(\tau) &= \Tr(N_\tau(P)), \quad \tau > 0 \\
	\Tr(P)(0) &= \Tr(\Nt(P)) + \Tr(N_\fzf(P)),
\end{aligned}
	\label{E:index_trace}
\end{equation}
provided all the operators on the right hand side are trace-class; see
Proposition~\ref{P:trace_at_zero}. This is used to deduce
\eqref{E:index_family} through the general identity $\ind(P) = \Tr([P,Q]) =
\Tr(R_L - R_R)$ for the index of a Fredholm operator $P$ given a parametrix $Q$
such that the error terms $R_L = I - QP$ and $R_R = I - PQ$ are trace-class.

\subsection{The case of zero rank} \label{S:index_rank0}
We now consider the operator $\Pfam = D + \Phi - i\chi \tau$ in the zero
rank case ($\Phi \rst_{\pa X} = 0$) within the framework of the transition
calculus. Before doing so we note two things: first, the transition
calculus is most naturally defined using b half densities rather than
scattering half densities, so we conjugate $\Pfam$ by $x^{n/2}$ in
order to make it act on sections of $\bOh.$ Second, we observe

\begin{lem}
The lift of $\Pfam$ to $\tX$ vanishes to first order at the boundary face $\ftf
\subset \tX$.
\label{L:Pfam_vanish_tf}
\end{lem}
\begin{proof}
The operator $D$ is a linear combination of scattering vector fields, and it
has already been shown in \eqref{E:lift_sc_to_trans} that these lift to vanish
to first order at $\ftf$. 
Furthermore, since we are assuming $\Phi = \cO(x^{1+\e'}) =
\cO((\tau\xi)^{1+\e'})$ it follows that, as a transition differential operator,
$\Pfam$ has an overall factor of $\tau$ near $\ftf\cap \fsc$ and $x$ near
$\ftf\cap \fzf$ so in fact $\Pfam \rst_{\ftf} = 0.$
\end{proof}

To resolve this degeneracy, we factor out a power of the boundary defining
function $\rho_{\ftf}$ for $\ftf$.  The choice of such a $\rho_\ftf$ matters
here (in the sense that different choices will lead to different normal
operators at $\ftf$, though these will all share the same essential analytical
properties), as does the manner in which it is factored out (compare the
discussion in \S \ref{S:fred_zero}). We will use a function satisfying
\begin{equation}
	\rho_\ftf = \begin{cases} x & \text{near $\ftf \cap \fzf$, and} \\ \tau & \text{near $\ftf \cap \fsc$.}
	\end{cases}
	\label{E:ftf_bdf_convention}
\end{equation}
and take
\begin{equation}
	\Pfamt = x^{-n/2} \rho_\ftf^{-1/2} \Pfam \rho_\ftf^{-1/2} x^{n/2}
	\label{E:index_conj_conv}
\end{equation}
(compare \eqref{E:fred_conj_conv}). With this convention $N_\fzf(\Pfamt)$
agrees with the operator $\wt P$ in \S \ref{S:fred_zero} and again has the
virtue of preserving formal self-adjointness of the Dirac operator. 

\begin{rmk}
Note that the convention \eqref{E:fred_conj_conv} can also be understood as
taking the boundary defining function $\rho_\ftf$ on the double space to
coincide with $(x\,x')^{1/2}$ near $\ftf \cap \fzf$, where $x$ and
$x'$ are pulled back from the left and right factors of $X$, respectively.
\end{rmk}

\begin{lem}
$\Pfamt \in \cB\tDiff^1(\tX; V\otimes \bOh)$ and
\begin{enumerate}
[{\normalfont (a)}]
\item $N_\fzf(\Pfamt) = \wt P$, where $\wt P$ is the operator in \S \ref{S:fred_zero}.
\item $\Nt(\Pfamt) = \ssD - i\phi$, where $\ssD$ is a 
first order operator, which is formally self-adjoint if $D$ is, and $\phi$ is a non-negative scalar function, strictly
positive on the interior, such that $\phi \rst_{\pa X \times 0} = 0$ and $\phi
\rst_{\pa X \times 1} = 1.$ The indicial families are related by
\begin{equation}
	I(\Nt(\Pfamt),\lambda) = I(\ssD,\lambda) = I(\wt D, -\lambda) = I(N_\fzf(\Pfamt),-\lambda)
	\label{E:index_normalop_indfam}
\end{equation}
\end{enumerate}
\label{L:index_normalops}
\end{lem}
\begin{proof}
The first claim has been discussed already and follows from the fact that the
normalization convention \eqref{E:index_conj_conv} agrees with
\eqref{E:fred_conj_conv}.

For the second claim, observe that the assumption $\Phi = \cO(x^{1+\e'})$ means
that $\rho_\ftf^{-1/2}\Phi \rho_\ftf^{-1/2} = \cO(\rho_\ftf^{\e'})$ still
vanishes at $\ftf$, so $N_\ftf(\Pfamt)$ is given by the restriction of
$x^{-n/2}\rho_\ftf^{-1/2}(D - i \chi\tau)\rho_\ftf^{-1/2}x^{n/2}.$

At the b end of $\ftf$ (i.e.\ $\ftf \cap \fzf$), the normalization convention
and \eqref{E:lift_sc_to_trans} show that $\ssD$ agrees with $\wt D =
x^{-(n+1)/2} D x^{(n-1)/2}$ apart from the replacement of $x\pa_x$ by $- \eta
\pa_\eta$ (giving a direct verification of the indicial relation $I(\ssD,
\lambda) = I(\wt D, - \lambda)$).  Near the scattering end,
\eqref{E:lift_sc_to_trans} shows that $\ssD$ agrees with $x^{-n/2}Dx^{n/2}$
apart from the replacement of $x$ by $\xi.$

Finally, it follows from \eqref{E:ftf_bdf_convention} that $\phi$, which is the
restriction of $x^{-n/2}\rho_\ftf^{-1/2} \tau\chi \rho_\ftf^{-1/2}x^{n/2} =
\rho_\ftf^{-1/2} \tau\chi \rho_\ftf^{-1/2}$ agrees locally with boundary
defining function $\eta$ near the b end and with $1$ near the scattering end.
\end{proof}

Along with the evident invertibility of the scattering symbol of $\Nt(\Pfamt)$,
given explicitly by $\ssym(x^{-n/2} D x^{n/2}) - i\id = \ssym(D) - i\id$, the
indicial family relations imply
\begin{cor}
$\Nt(\Pfamt)$ admits Fredholm extensions 
\begin{multline}
	\Nt(\Pfamt) : \rho_\mathrm{b}^\alpha \rho_\mathrm{sc}^\eta \bscH^{k}(\pa X\times \bIsc; V\otimes \bOh)
	\\ \to \rho_\mathrm{b}^\alpha \rho_\mathrm{sc}^\eta \bscH^{k-1}(\pa X\times \bIsc; V\otimes \bOh)
	\label{E:index_transop_fred_general}
\end{multline}
for all $k$, $\eta$ and $\alpha \notin \bspec(\ssD) = - \bspec(\wt D).$
\label{C:index_transop_fred}
\end{cor}
\begin{rmk}
Here it should be noted that we're using the notation $\bscH^\ast$ in a
different way than used previously, to mean regularity with respect to
derivatives which are of b type near $\pa X \times 0$ and scattering type
near $\pa X\times 1.$
\end{rmk}

Using the relative index theorem of \cite{melrose1993atiyah}, it is possible
to show directly that the index of \eqref{E:index_transop_fred_general}
satisfies
\begin{equation}
	\ind(\Nt(\Pfamt),\alpha) = - \ind(N_\fzf(\Pfamt), -\alpha). 
	\label{E:index_transop_reln}
\end{equation}
However, in order to generalize this result to the setting of Callias operators
with general rank potentials when the splitting $V = V_0\oplus V_1$ does not
extend globally across the manifold, we derive \eqref{E:index_transop_reln} 
as a consequence of the following parametrix construction for $\Pfamt$:

\begin{prop}
Under the assumption that $\Phi = \cO(x^{1+\e'})$, for any choice of $\alpha
\notin \bspec(\wt D)$ there exists a parametrix $\Qfam \in \tP^{-1,\cE}(\tX;
V\otimes \bOh)$, such that 
\begin{equation}
	\RfamL = \id - \Qfam\Pfamt \in \tP^{-\infty,\cF^L},\quad \RfamR = \id - \Pfamt\Qfam \in \tP^{-\infty,\cF^R}
	\label{E:index_mainprop_error}
\end{equation}
are trace class, with continuous trace in $\tau.$ Here
\[
\begin{aligned}
	\cE &= (0_\fsc, 0_\ftf, E^+(\alpha)_{\flb_0}, E^-(\alpha)_{\frb_0}), \\
	\cF^L &= (\infty_\fsc, 0_\ftf, E^+(\alpha)_{\flb_0}, E^-(\alpha)_{\frb_0}) \\
	\cF^R &= (\infty_\fsc, 0_\ftf, E^+(\alpha)_{\flb_0}, E^-(\alpha)_{\frb_0})
\end{aligned}
\]
where $E^\pm(\alpha) = \set{(\pm z +n,l) : (z,l) \in \bspec(\wt D),\ z \gele
\alpha,\ n \in \bbN_0}.$

In particular, $N_\fzf(\Qfam)$ is a Fredholm inverse for the extension
\begin{equation}
	 N_\fzf(\Pfamt) : x^{\alpha} \bH^k(X; V\otimes \bOh) \to x^{\alpha}\bH^{k-1}(X; V\otimes \bOh),
	\label{E:index_mainprop_zf}
\end{equation}
and $\Nt(\Qfam)$ is a Fredholm inverse for the extension
\begin{multline}
	\Nt(\Pfamt) : \rho_\mathrm{b}^{-\alpha} \bscH^{k,l}(\pa X\times \bIsc; V\otimes \bOh)
	\\\to \rho_\mathrm{b}^{-\alpha} \bscH^{k-1,l-1}(\pa X\times \bIsc; V\otimes \bOh).
	\label{E:index_mainprop_tf}
\end{multline}
Finally,
\begin{equation}
	\Tr(\RfamL - \RfamR)(0) = \ind(N_\fzf(\Pfamt), \alpha) + \ind(\Nt(\Pfamt),-\alpha) = 0.
	\label{E:index_transop_reln_proved}
\end{equation}
\label{P:index_mainprop}
\end{prop}
\noindent To say that the remainder terms are trace class here means that
$N_\tau(\RfamRL)$, $\Nt(\RfamRL)$ and $N_\fzf(\RfamRL)$ are each trace class,
so in particular \eqref{E:index_trace} holds.
\begin{proof}
As $\Pfamt$ is elliptic, its principal symbol $\sigma_1(\Pfamt)$ is invertible,
so we may initially choose $\Qfam_0 \in \tP^{-1,(0,0,\infty,\infty)}$ such that
$\sigma_{-1}(\Qfam_0) = \sigma_1(\Pfamt)^{-1}$, and by the usual iterative
argument, we may assume that $\cR_0 = I - \Pfamt\Qfam_0 \in
\tP^{-\infty,(0,0,\infty,\infty)}$.

To improve this further to a Fredholm parametrix for $\Nt(\Pfamt)$ and
$N_\fzf(\Pfamt)$, we next add a term $\Qfam_1 \in \tP^{-\infty,\cE}$ such that
$\ssym(\Qfam_1) = \ssym(\Pfamt)^{-1}\ssym(\cR_0)$ and
\[
\begin{aligned}
	I(N_\fzf(\Qfam_1)) \equiv I(\Nt(\Qfam_0))
	&= \cM^{-1}_{\re\lambda = \alpha} I(N_\fzf(\Pfamt),\lambda)^{-1}I(N_{\fzf}(\cR_0),\lambda)
	\\&= \cM^{-1}_{\re\lambda = \alpha} I(\Nt(\Pfamt),-\lambda)^{-1}I(\Nt(\cR_0),-\lambda).
\end{aligned}
\]
Since $\ssym(\Pfamt) \rst_{\fsc \cap \ftf} = \ssym\bpns{\Nt(\Pfamt)}$, 
it follows that $\Nt(\Qfam_0 + \Qfam_1)$ is a Fredholm parametrix for $\Nt(\Pfamt)$
as an operator \eqref{E:index_mainprop_tf}, and likewise $N_\fzf(\Qfam_0 + \Qfam_1)$
is a Fredholm parametrix for $N_\fzf(\Pfamt)$ as an operator \eqref{E:index_mainprop_zf}.
The order $\cE = (0_\fsc, 0_\ftf, E^+(\alpha)_{\flb_0}, E^-(\alpha)_{\frb_0})$
is a standard consequence of taking the inverse Mellin transform along $\re
\lambda = \alpha$.

Finally, to produce a parametrix with trace-class error terms in the transition
calculus, we inductively produce operators $\Qfam_i \in \tP^{-\infty,\cE_i}$,
$\cE_i = (\gamma_{i-1,\fsc}, 0_\ftf,\infty_\flbz,\infty_\frbz)$ where
$\gamma_{i-1}$ denote successive terms in the expansion of $\Pfamt$ at $\fsc$
(in the smooth case, $\gamma_{i-1} = i-1$), supported in a small neighborhood
of $\fsc$, satisfying 
\[
\begin{gathered}
	\ssym(\Qfam_i) = \ssym(\Pfamt)^{-1}\ssym(\cR_{i-1}), \\
	\cR_{i-1} = I - \Pfamt(\Qfam_0 + \cdots + \Qfam_{i-1}) \in \tP^{-\infty,(\gamma_{i-1},0,E^+(\alpha),E^-(\alpha))}.
\end{gathered}
\]
Asymptotically summing the resulting series results in the parametrix $\Qfam
\in \tP^{-1,\cE}(X; V\otimes \bOh)$ as claimed, with error terms
\eqref{E:index_mainprop_error}.  In particular, it follows that the remainders
\[
\begin{aligned}
	\Nt(\RfamL) &= I - \Nt(\Qfam)\Nt(\Pfamt), & \Nt(\RfamR) &= I - \Nt(\Pfamt)\Nt(\Qfam) \\
	N_\fzf(\RfamL) &= I - N_\fzf(\Qfam)N_\fzf(\Pfamt) & N_\fzf(\RfamR) &= I - N_\fzf(\Pfamt)N_\fzf(\Qfam)
\end{aligned}
\]
are each trace class. Once these are known to be trace class, the families
trace \eqref{E:families_trace} is continuous by
Proposition~\ref{P:trace_at_zero}. Finally, equation
\eqref{E:index_transop_reln_proved} follows by expressing
$\ind\big(N_\fzf(\Pfamt),\alpha\big)$ and $\ind\big(\Nt(\Pfamt), - \alpha\big)$
as the traces of the differences  $N_\ast(\cR_L) - N_\ast(\cR_R)$, and using
the fact that $\ind(N_\tau(\Pfamt)) = 0$ for $\tau > 0$ by
Theorem~\ref{T:calliasthm}, since the potential $- i \chi \tau$ has trivial
$+i$ eigenbundle $V_+$.
\end{proof}

From Proposition~\ref{P:fred_rank0_noext} and \eqref{E:index_transop_reln_proved}
we obtain
\begin{cor}
The indices of the extensions \eqref{E:index_mainprop_tf} of $\Nt(\Pfamt)$ 
satisfy
\begin{equation}
\begin{aligned}
	\ind(\Nt(\Pfamt),-\alpha) &= - \ind(\Nt(\Pfamt),\alpha), \quad \text{if $D$ is self-adjoint, and}\\
	\ind(\Nt(\Pfamt),\alpha_0 - \epsilon) &= \ind(\Nt(\Pfamt),\alpha_0+\epsilon) + \dim F(\wt D,-\alpha_0)
\end{aligned}
	\label{E:transop_index_props}
\end{equation}
for $\alpha_0 \in \bspec(\wt D_0)$ and $\epsilon > 0$ sufficiently small.
\label{C:index_mainprop}
\end{cor}

\subsection{The case of general rank} \label{S:index_genrank}
We may now compute the index of the Fredholm extensions
\eqref{E:fredholm_hybrid_general}. We recall the operator $\pa^+_+ \in
\Diff^1(\pa X; V^+_+\otimes \scOh, V^-_+\otimes \scOh)$ induced by $D$ at $\pa
X$, acting on the bundle of positive imaginary eigenvectors $V_+$ of $\Phi
\rst_{\pa X}$, discussed prior to Theorem~\ref{T:calliasthm}.

\begin{thm}
The Fredholm extension \eqref{E:fredholm_hybrid_general} has index
\begin{equation}
	\ind(P) = \ind(\pa^+_+) + \defect(P,\alpha).
	\label{E:hybrid_index_formula}
\end{equation}
The defect term $\defect(P,\alpha) \in \Z$ is constant for $\alpha \notin
\bspec(\wt D_0)$ and satisfies
\begin{equation}
\begin{aligned}
	\defect(P,- \alpha) &= - \defect(P,\alpha), \quad \text{if $D$ is self-adjoint, and} \\
	\defect(P,\alpha_0-\epsilon) &= \defect(P,\alpha_0 + \epsilon) + \dim F(\wt D_0,\alpha_0)
\end{aligned}
	\label{E:hybrid_index_defect_props}
\end{equation}
for $\alpha_0 \in \bspec(\wt D_0)$ and sufficiently small $\epsilon > 0.$
\label{T:hybrid_index_thm}
\end{thm}
\begin{proof}
The strategy is to combine the parametrix constructions in the proof of
Theorem~\ref{T:fredholm_hybrid_norealroot} and the proof of
Proposition~\ref{P:index_mainprop}.  
Since Fredholm operators form an open set, and the off-diagonal terms in
\eqref{E:fred_P_splitting} vanish at $\pa X$, we can, making an arbitrarily
small norm perturbation of $P$ as an operator $P :
\cH^{\alpha-1/2,\alpha+1/2,k,1} \to \cH^{\alpha+1/2,\alpha+1/2,k,0}$,
remove the off-diagonal terms on a small collar neighborhood of $\pa X$. Thus
it suffices to assume that $P$ is diagonal with respect to $V = V_0\oplus V_1$
in a neighborhood $U$ of $\pa X.$ 
Then letting $\Pfam = D + \Phi - i\chi \tau \oplus 0$ as in the beginning of this
section, it follows that
\[
	\Pfam \rst_{U} = \begin{pmatrix} D_0 - i \chi \tau & 0 \\ 0 & D_1 + \Phi \end{pmatrix}.
\]

Fixing $\alpha \notin \bspec(\wt D_0),$ we construct a parametrix $\cQ \in
C^{-\infty}(X^2 \times I; \End(V)\otimes \scOh)$ for $\Pfam$ as a distribution
supported near and conormal to the fiber diagonal in $X^2 \times I$, such that
\[
\begin{gathered}
	\cQ \rst_{U^2\times I} = \begin{pmatrix} \cQ_0 & 0 \\ 0 & \cQ_1 \end{pmatrix},
	\quad \cQ_0 = x^{n/2}\rho_\ftf^{-1/2}\wt\cQ_0{\rho'}_{\ftf}^{-1/2} {x'}^{-n/2} \\
	\wt\cQ_0 \in \tP^\ast(U\times I; V_0\otimes \scOh), \quad \cQ_1 \in C^\infty\big(I; \scP^*(U; V_1\otimes \scOh)\big)
\end{gathered}
\]
with $x,\rho_\ftf$ and $x',\rho'_\ftf$ denoting the lifts from the left and
right, respectively.

First let $\sigma_{-1}(\cQ^0_1(t)) = \sigma_{1}(D_1)^{-1}$ for all $t \in I$ and
$\sigma_{-1}(\wt \cQ^0_0) = \sigma_1(\wt D_0)^{-1}$, which are compatible with an
overall choice such that $\sigma_{-1}(\cQ^0) = \sigma_1(\Pfam)^{-1}$ in the
interior. By iterating and summing, we may suppose that $\cQ^0$ removes the diagonal
singularity of $\cP$ to infinite order, i.e.\ that $\cR^0 = I - \cP\cQ^0$ is smooth
across the diagonal.

We then improve this by choosing $\cQ^1$ so that $\ssym(\cQ_1^1(t)) = \ssym(D_1 + \Phi)^{-1}\ssym(\cR^0_1(t))$,
$\ssym(\wt \cQ^1_0) = \ssym(\wt \cP_0)^{-1}\ssym(\wt \cR^0_0)$, and
\[
	I(N_\fzf(\wt \cQ^1_0) = \cM^{-1}_{\re \lambda = \alpha} I(N_\fzf(\wt \cP_0),\lambda)^{-1}I(N_\fzf(\wt \cR^0_0),\lambda).
\]
Finally, we iteratively produce $\cQ^i$ so that $\ssym(\cQ^i_1(t)) =
\ssym(D_1 + \Phi)^{-1}\ssym(\cR^{i-1}_1(t))$ and $\ssym(\wt \cQ^i_0) = \ssym(\wt
\cP_0)^{-1}\ssym(\wt \cR^{i-1}_0)$, where $\cR^i = I - \cP(\cQ^0 + \cdots +
\cQ^i).$ Asymptotically summing the resulting series, we denote the result by
$\cQ.$

The remainder terms $\RfamL = I - \cQ \Pfam$ and $\RfamR = I - \Pfam\cQ$ are
smooth on the interior of $X^2\times I$ and are diagonal with respect to
$V_0\oplus V_1$ near $\pa X^2\times I$. For fixed $\tau > 0$ these are trace
class operators on any $x^\alpha \scH^k(X; V\otimes \scOh)$, and the families
trace --- given by restriction to the fiber diagonal $X\times I \subset
X^2\times I$ followed by push forward to $I$ --- of $\RfamL$ and $\RfamR$ 
is continuous by the considerations of the previous section. 

At $\tau = 0$, $\Tr\pns\RfamL$ has the form\footnote{Evaluation at the diagonal has the effect
of removing terms of the form $\rho_\ftf^{-1/2} {\rho'}_\ftf^{1/2}$ and
$x^{n/2}{x'}^{-n/2}$ which would otherwise be present}
\begin{multline*}
	\Tr(\RfamL)(0) = \Tr\begin{pmatrix} \Tr \bpns{I_\ftf - \Nt(\wt \cQ_0) \Nt(\wt \cP_0)} & 0 \\ 0 & 0 \end{pmatrix} 
	\\+ \Tr\begin{pmatrix} \Tr \bpns{I_\fzf - N_\fzf(\wt \cQ_0)N_\fzf(\wt \cP_0)} & 0 \\ 0 & \Tr\bpns{I - \cQ_1(0)(D_1 + \Phi)} \end{pmatrix}
\end{multline*}
and similarly for $\Tr(\RfamR)(0),$ where here we recall that $\Pfamt_0 =
x^{n/2}\rho_\ftf^{-1/2} (D_0 + i\chi \tau){\rho'}_\ftf^{-1/2} {x'}^{n/2}$. We
observe that the rightmost term above is the error for a Fredholm parametrix
for the extension \eqref{E:fredholm_hybrid_general} as constructed in the proof
of Theorem~\ref{T:fredholm_hybrid_norealroot}.

From the trace formula for the index it follows that
\[
	\Tr\bpns{\RfamL - \RfamR}(\tau) = \ind(\Pfam) = \ind(\pa^+_+), \quad \tau > 0
\]
using Theorem~\ref{T:calliasthm} and the fact that the positive imaginary
eigenbundle $V_+$ of $\Phi - i\chi\tau \oplus 0$ coincides with that of $\Phi$.
Then from the constancy of the trace of $\cR_{L/R}$, it follows that
\[
	\ind(\pa^+_+) = \Tr\bpns{\RfamL - \RfamR}(0) = \ind(\Nt(\Pfamt_0), - \alpha) + \ind(P)
\]
where $\ind(P)$ denotes the index of the extension
\eqref{E:fredholm_hybrid_general}. The proof is completed by defining
\[
	\defect(P,\alpha) := -\ind(\Nt(\Pfamt_0),-\alpha) 
\]
which satisfies \eqref{E:hybrid_index_defect_props} as a consequence of
Corollary~\ref{C:index_mainprop}. 
\end{proof}

\appendix
\section{Polyhomogeneity} \label{S:phg}
Here we summarize some standard results regarding polyhomogeneous conormal
distributions. References for the material this section include
\cite{melrose1993atiyah} and \cite{melrose1992calculus}. We say $E \subset \C
\times \N$ (here $\N = \set{0,1,\cdots}$) is an {\em index set} if it is
discrete and satisfies $\re (z_j) \to \infty$ when $\abs{(z_j,k_j)} \to \infty$ for
$(z_j,k_j) \in E$.

Let $X$ be a manifold with boundary and boundary defining function $x$. We say
$u \in C^\infty(X \setminus \pa X)$ has {\em polyhomogeneous expansion in $x$} with
index set $E$ if $u$ is asymptotic to the sum
\[
	u \sim \sum_{(z,k) \in E} x^z(\log x)^k u_{z,k}
\]
at $\pa X$, with $u_{z,k} \in C^\infty(\pa X)$.  The condition that $\re z_j
\to \infty$ guarantees that such sums are Borel summable. Such an expansion
is dependent on the choice of $x$, but if we require that $E$ is a {\em smooth
index set}, meaning in addition 
\[
	(z,k) \in E \implies (z+n,l) \in E, \quad \forall n \in \N,\  0 \leq l \leq k,
\]
then the notion that $u$ has a polyhomogeneous expansion with index set $E$ is
independent of the choice of $x$, and we let $\Aphg^E(X)$ denote the set of
such functions. Likewise if $V \to X$ is a vector bundle, the set
$\Aphg^E(X;V)$ is well-defined for smooth $E$ and consists of sections with
expansions in terms of some local frame.

The presence or absence of a term associated to a particular $(z,k) \in E$ in
the expansion of $u$ generally depends on the choice of boundary defining
function, but there are certain terms for which the vanishing or non-vanishing
of the corresponding coefficient makes coordinate-invariant sense. We say $(z,k)
\in E$ is {\em high order} if $z = \min \set{z + n : n \in \Z,\ (z,l) \in E}$
and $k = \max \set{l : (z,l) \in E}.$ It follows that if $u_{(z,k)} \neq 0$ for
a high order term with respect to one choice of $x$, then the same coefficient
will be nonzero in the expansion with respect to any other choice, and
$u_{(z,k)} = 0$ is equivalent to $u \in \Aphg^{E\setminus (z,k)}(X).$

For a fixed $x$, polyhomogeneity with index set $E$ is equivalent to the {\em
Mellin transform}
\[
	\cM(\phi u)(\lambda) = \int_{\R_+} x^{-\lambda}\phi(x) u(x) \frac {dx}{x}
\]
being a meromorphic function (section of $V \rst_{\pa X}$) with respect to
$\lambda \in \C$ with poles at $\lambda = z$ of order $k+1$ where $k = \max\set{l :
(z,l) \in E},$ and with rapid decay uniformly in strips $\abs{\re \lambda} \leq
c$ as $\abs{\im \lambda} \to \infty$. Here $\phi \equiv 1$ on a neighborhood of
$\pa X$ with support in a larger neighborhood, and $\cM(\phi u) - \cM(\phi' u)$
is holomorphic for any other such $\phi'$.

We make use of the following notation for index sets. We identify $e \in \Z$
with the index set $\set{(e + \N_0,0)}$, and write $\infty$ for the empty index
set (since it corresponds to functions vanishing to infinite order). We also write
\[
	\re E := \inf \set{ \re \pns z : (z,k) \in E}, \quad \im E := \inf \set{ \im \pns z : (z,k) \in E}.
\]

Now suppose $X$ is a manifold with corners and let $\cM_l(X)$ be the set of its
boundary faces of codimension $l$; in particular $\cM_1(X)$ is the set of
boundary hypersurfaces. We use the notation $\cE = \pns{E_H : H \in \cM_1(X)}$
for a multi-index of smooth index sets, and the space $\Aphg^\cE(X;V)$ is
defined recursively by
\[
	u \in \Aphg^\cE(X;V) \iff u \sim \sum_{(z,k) \in E_H} \rho_H^z(\log
	\rho_H)^k u_{z,k}, \quad u_{z,k} \in \Aphg^{\cE(H)}(H;V)
\]
where $\cE(H) = \pns{F_G : G \in \cM_1(H)}$ and $F_G = E_{H'}$ for the unique
$H' \in \cM_1(X)$ such that $H \cap H' = G.$ For a closed manifold $Y$,
$\Aphg^\ast(Y) \equiv C^\infty(Y)$, so that the recursion eventually terminates
and $\Aphg^\cE(X;V)$ is well-defined.

For any boundary face $F \in \cM_n(M)$ and a normal neighborhood $U \cong
F\times[0,1)^n$ defined in terms of fixed boundary defining functions
$(x_1,\ldots,x_n)$ for hypersurfaces $H_1,\ldots,H_n$ such that $F \subset H_1
\cap \cdots \cap H_n,$ the {\em multi-Mellin transform} of $\phi\,u \rst_U$
(where $\phi$ is a compactly supported cutoff function in $U$ nowhere vanishing
on $F$) is defined by
\[
	\cM_F(\phi\,u)(z,\lambda_1,\ldots,\lambda_n) = \int_{\R_+^n} x_1^{-\lambda_1}\cdots x_n^{-\lambda_n}\,\phi\,u,\frac{dx_1}{x_1}\ldots\frac{dx_n}{x_n},
\]
where $(z,\lambda_1,\ldots,\lambda_n) \in F\times \C^n$.  Then $u$ is
polyhomogeneous with multi-index $\cE$ if and only if, for each such boundary
face $F$, $\cM_F(\phi\,u)$ is a product of meromorphic functions valued in
$\Aphg^\ast(F;V)$, with factors having poles of order $k_i + 1$ at
$\set{\lambda_i = z_i}$ only if $(z_i,k_i) \in E_{H_i}$, and decaying rapidly
and uniformly in strips $\abs{\re \lambda_i} \leq c.$ For a fixed choice of
defining functions $x_i$, a particular pole may or may not occur according
to whether the corresponding term appears in the expansion in terms of the
$x_i$. However the poles corresponding to high order terms are fundamental,
and if such a pole does not occur then $u$ is polyhomogeneous for the
multi-index in which the associated element has been removed.

\subsection{Pullback and pushforward} 

Fundamental to the use of polyhomogeneous distributions on manifolds with
corners are two results dictating their behavior with respect to pullback and
pushforward operations.  We say a map $f : X \to Y$ between manifolds with
corners is a {\em b-map} provided that for each boundary defining function
$\rho_G$ for $G \in \cM_1(Y)$,
\[
	f^\ast(\rho_G) = a \prod_{H \in \cM_1(X)} \rho_H^{e(H,G)}, \quad e(H,G) \in \N,\ 0 < a \in C^\infty(X),
\]
and b-maps give a well-defined set of morphisms with respect to which compact
manifolds with corners form a category.  The {\em boundary exponents} $e(H,G)
\in \N$ do not depend on the choice of the $\rho_G$ or $\rho_H$. A b-map $f$ is
said to be a {\em b-fibration} if its natural differential ${}^\mathrm{b} f_* :
\bT X \to \bT Y$ is everywhere surjective, and for each $H \in \cM_1(X)$,
$e(H,G) \neq 0$ for at most one $G \in \cM_1(Y)$. Such a map restricts to a
fibration between the interiors of $X$ and $Y$, and restricted to any boundary
face of $X$, $f$ is again a b-fibration.
 
Suppose $f : X \to Y$ is a b-map with boundary exponents $e(H,G)$. For a
collection $\cF = \set{F_G}_{G \in \cM_1(Y)}$ of smooth index sets for $Y$, we
define the following pullback operation on index sets:
\[
\begin{aligned}
	f^\#\cF &= \set{E_H}_{H \in \cM_1(X)}, \quad \text{where} \\
	E_H &= \set{\textstyle\sum_{e(H,G) \neq 0} \pns{e(H,G)z_i,k_i} : (z_i,k_i) \in F_G}.
\end{aligned}
\]

If $f$ is a b-fibration, we define a pushforward operation by
\[
\begin{aligned}
	f_\#(\cE) &= \set{F_G}_{G \in \cM_1(Y)}, \quad \text{where} \\
	F_G &= \ol \bigcup_{H \subset f^{-1}(G)} \set{\pns{z/e(H,G), k} : (z,k) \in E_H},
\end{aligned}
\]
and where the {\em extended union operation} is defined on index sets as 
\[
	E \ol\cup F = E \cup F \cup \set{(z,k+l+1) : (z,k) \in E \text{ and } (z,l) \in F}.
\]

The following results can be found in \cite{melrose1992calculus}.
\begin{thm} \label{T:man:pullback}
Let $f : X \to Y$ be a b-map.  Then the pullback $f^\ast : \dot C^\infty(Y;V) \to
\dot C^\infty(X;f^\ast V)$ (here $\dot C^\infty$ denotes smooth functions vanishing to infinite order at all boundary faces) extends to a map 
\[
	f^\ast : \Aphg^\cF(Y; V) \to \Aphg^{f^\#\cF}(X; f^\ast V).
\]
\end{thm}
\noindent The corresponding pushforward result requires the use of densities.
\begin{thm} \label{T:man:pushforward}
Let $f : X \to Y$ be a b-fibration.  If $\re(E_H) > 0$ for all $H$ such that $f(H) \cap
\mathring Y \neq \emptyset$, then the pushforward map $f_\ast : \dot
C^\infty(X; \bO(X)) \to \dot C^\infty(Y; \bO(Y))$ extends to a map
\[
	f_\ast : \Aphg^\cE\bpns{X; \bO(X)} \to \Aphg^{f_\#\cE}\bpns{Y; \bO(Y)}.
\]
\end{thm}
\begin{rmk}
In terms of the multi-Mellin characterization of polyhomogeneity, the extended
union corresponds to multiplication of meromorphic functions and the
addition of the degrees of their poles where they align. (See the proof of
Proposition~\ref{P:trace_at_zero}.)
\end{rmk}

\section{b-sc transition calculus} \label{S:transition}
We include here a self-contained summary of the b-sc transition calculus of
pseudodifferential operators.  This idea is due to Melrose and Sa Barretto, and
was used by Guillarmou and Hassell in \cite{guillarmou2008resolvent}.

Let $X$ be a compact manifold with boundary, and let $\bV(X)$, $\scV(X)$ denote
the Lie algebras of b vector fields and scattering vector fields, respectively.
Let $I = [0, 1)$ be a half open interval.  The calculus is meant to
microlocalize families of differential operators, parameterized by $\tau \in I$,
which fail to be fully elliptic in the scattering sense precisely as $\tau \smallto
0$, where they are treated as weighted b type operators.

\subsection{Single, Double, and Triple Spaces} 
\label{S:calculus_spaces}
The operators are constructed as
Schwartz kernels on $X^2\times I$, acting on functions on $X\times I$, with
composition occurring on $X^3\times I$ all via pullback, multiplication and
pushforward. The operators and functions considered will be those resolved to have
polyhomogeneous expansions by particular blow-ups of these spaces.

The {\em single space} is defined by
\[
	\tX = [X\times I; \pa X \times \set{0}].
\]
The boundary faces of $\tX$ are denoted as follows.
\begin{align*}
	\fsc &= \text{lift of $\pa X \times I$} \\
	\ftf &= \text{lift of $\pa X \times \set{0}$} \\
	\fzf &= \text{lift of $X \times \set{0}$}
\end{align*}
See Figure \ref{F:singlespace} on page \pageref{F:singlespace}. We recall that the {\em lift} of a submanifold
$Y \subset Z$ to the blow up $[Z;W]$ is defined as follows. If $Y \subset W$
then the lift of $Y$ is just the inverse image $\beta^{-1}(Y)$ under the
blow-down map $\beta : [Z;W] \to Z$, while if $Y \setminus W \neq \emptyset$
the lift is defined to be the closure $\overline{\beta^{-1}(Y \setminus W)}$.


The double space is defined in two steps.  Let $C_n$ denote the union of
boundary faces of codimension $n$ of $X^2 \times I$.  Thus $C_3 = \pa X \times
\pa X \times \set{0}$, while $C_2$ is a union of the faces $\pa X \times \pa X
\times I$, $\pa X \times X \times \set{0}$ and $X \times \pa X \times \set{0}$.
The {\em b blowup} or {\em total boundary blowup} is well-defined for any
manifold with corners to be the blow up of all boundary faces in order of
decreasing codimension. In this case,
\[
	(X^2\times I)_\mathrm{b} = [X^2\times I; C_3, C_2] = [[X^2\times I; C_3]; \text{ lift of } C_2].
\]
Recall that when performing multiple blow-ups, the order is generally
important, so part of the assertion that such a space is well-defined is to say
that it does not depend on the order in which the faces in $C_2$ are blown up,
for instance.

Now denote by $C_V$ and $ \Delta$ the lifts of $\pa X
\times \pa X \times I$ and the fiber diagonal $\Delta \times I$, respectively.
These intersect transversally in $(X^2\times I)_\mathrm{b}$ and the 
{\em double space} is defined by
\[
	\tX^2 = [(X^2\times I)_\mathrm{b}; C_V \cap \Delta]
\]
We denote by $\beta_2 : \tX^2 \to X^2\times I$ the composite blow down map,
and denote the boundary faces of $\tX^2$ by 
\begin{align*}
	\fsc &= \text{lift of $C_V \cap \Delta$}&
	\fbf &= \text{lift of $\pa X \times \pa X \times I$}\\
	\ftf &= \text{lift of $C_3$} &
	\flb_0 &= \text{lift of $X \times \pa X \times \set{0}$} \\
	\frb_0 &= \text{lift of $\pa X \times X \times \set{0}$} &
	\flb &= \text{lift of $X \times \pa X \times I$}\\
	\frb &= \text{lift of $\pa X \times X \times I$}&
	\fzf &= \text{lift of $X \times X \times \pa I$}
\end{align*}
See Figure \ref{F:doublespace}. The reuse of the names $\fsc$, $\ftf$ and
$\fzf$ should not cause any confusion as it should be clear from context which
space is being referred to. Observe that these three faces of $\tX^2$ coincide
with those of $\tX$ upon identifying $\tX$ with the lifted diagonal.

\begin{figure}[htb]
\begin{center}
\begin{tikzpicture}[scale=0.7]
	\path (-2,0) coordinate (A) +(-4,-2) coordinate (A1)
		++(1,2) coordinate (B) +(0,4) coordinate (B1)
		++(1,0.5) coordinate (C) +(0,4) coordinate (C1)
		++(1,0) coordinate (D) +(0,4) coordinate (D1)
		++(1,-0.5) coordinate (E) +(0,4) coordinate (E1)
		++(1,-2) coordinate (F) +(4,-2) coordinate (F1)
		++(-1,-2) coordinate (G) +(4,-2) coordinate (G1)
		++(-3,0) coordinate (H) +(-4,-2) coordinate (H1);
	\draw [thick] (A) 
		  .. controls ($(A)!0.25!35:(B)$) and ($(B)!0.25!-35:(A)$) .. (B)  
		  .. controls ($(B)!0.25!35:(C)$) and ($(C)!0.25!-35:(B)$) .. (C)  
		  .. controls ($(C)!0.25!35:(D)$) and ($(D)!0.25!-35:(C)$) .. (D)  
		  .. controls ($(D)!0.25!35:(E)$) and ($(E)!0.25!-35:(D)$) .. (E)  
		  .. controls ($(E)!0.25!35:(F)$) and ($(F)!0.25!-35:(E)$) .. (F)  
		  .. controls ($(F)!0.25!-35:(G)$) and ($(G)!0.25!35:(F)$) .. (G)  
		  .. controls ($(G)!0.25!35:(H)$) and ($(H)!0.25!-35:(G)$) .. (H)  
		  .. controls ($(H)!0.25!-35:(A)$) and ($(A)!0.25!35:(H)$) .. (A);
	\draw [thick]
		(A1) -- (A) 
		(B1) -- (B)
		(C1) -- (C)
		(D1) -- (D)
		(E1) -- (E)
		(F1) -- (F)
		(G1) -- (G)
		(H1) -- (H);
	\coordinate (diag1) at ($0.5*(C1)+0.5*(D1)$);
	\coordinate (diag2) at ($0.5*(C) + 0.5*(D) + (0,0.125)$);
	\coordinate (diag3) at ($0.5*(H)+0.5*(G) + (0,-0.25)$);
	\draw [thick,dashed]
		(diag1) -- (diag2)
		.. controls ($(diag2)!0.25!-15:(diag3)$) and ($(diag3)!0.25!15:(diag2)$) .. (diag3)
		-- ($(diag3)+(-0.25,-2)$);
	\node  [right=0.2cm] at ($0.25*(H1)+0.25*(G1) + 0.25*(H)+0.25*(G)$) {$\fzf$};
	\node  at ($0.25*(A1)+0.25*(H1) + 0.25*(H)+0.25*(A)$) {$\flbz$};
	\node  at ($0.25*(F1)+0.25*(G1) + 0.25*(F)+0.25*(G)$) {$\frbz$};
	\node [right=0.2cm]  at ($(A)!0.5!(F)$) {$\ftf$};
	\node at ($(A) + (-2,2)$) {$\flb$};	
	\node at ($(F) + (2,2)$) {$\frb$};	
	\node  at ($0.25*(B1)+0.25*(C1) + 0.25*(B)+0.25*(C)$) {$\fbf$};
	\node  at ($0.25*(D1)+0.25*(E1) + 0.25*(D)+0.25*(E)$) {$\fbf$};
	\node [fill=white] at ($0.25*(D1)+0.25*(C1) + 0.25*(D)+0.25*(C)$) {$\fsc$};
\end{tikzpicture}
\caption{The double space $\tX^2$ and its boundary faces}
\label{F:doublespace}
\end{center}
\end{figure}
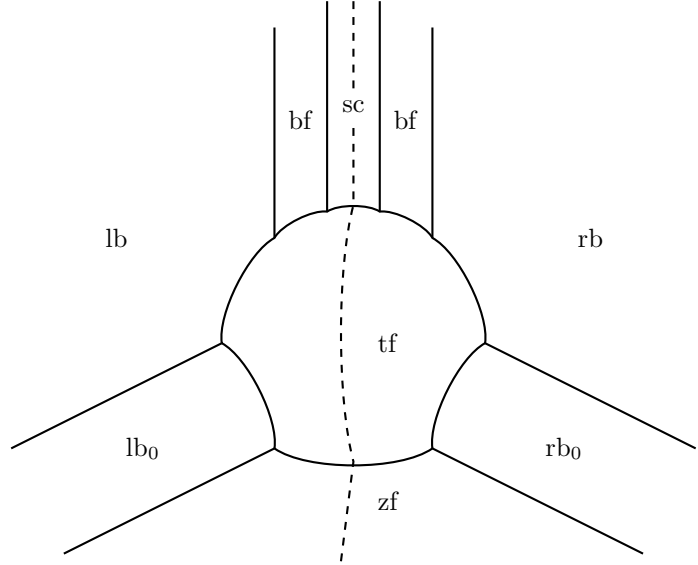

The triple space is similarly defined in two steps.  Again letting $C_n$ denote
the union of boundary faces of codimension $n$, the b-blowup is given by
\[
	(X^3\times I)_\mathrm{b} = [M^3; C_4, C_3, C_2].
\]
Now consider the b-fibrations $\pi_{L/R/C} : (X^3\times I)_\mathrm{b} \to
(X^2\times I)_b$.  The double space $\tX^2$ was obtained from $(X^2 \times
I)_\mathrm{b}$ by blowing up the intersection $C_V \cap \Delta$, which has
preimage under each of $\pi_L$, $\pi_R$ and $\Pi_C$ lying in two separate
boundary faces of $(X^3\times I)_\mathrm{b}$. Let $G_L$ denote the preimage of
$C_V \cap \Delta$ with respect to $\pi_L$ intersecting the preimage of $(\pa X
\times X^2 \times I)$ in $(X^3 \times I)_\mathrm{b}$, and let $J_L$ denote the
preimage of $C_V \cap \Delta$ intersecting the preimage of $(\mathring X\times
X^2\times I)$.  Define $G_{R/C}$ and $J_{R/C}$ similarly. A moment's
consideration reveals that $G_L \cap G_R \cap G_C = K$ is a nonempty
submanifold, while the $J_\ast$ only intersect the corresponding $G_\ast$.  

Finally, define the triple space by
\[
	 \tX^3 = [(X^3 \times I)_\mathrm{b}; K, G_L, G_R, G_C, J_L, J_R, J_C].
\]
It is well-defined since the $G_\ast$ are separated after blowing up $K$.

The important features of these spaces is that they lift the obvious
projections to b-fibrations.  A proof of the following theorem can be found in
\cite{guillarmou2008resolvent}.
\begin{thm} \label{T:tra:b_fibns}
There are b-fibrations $\pi_\ast$ making the following diagram commute
\begin{equation}\label{E:tra:b_fibns}
\xymatrix{
	\tX^3 \ar@<1ex>[r]^{\pi_{L,R,C}}\ar[r]\ar@<-1ex>[r] \ar[d]^{\beta_3}& \tX^2 \ar@<0.6ex>[r]^{\pi_{L,R}}\ar@<-0.6ex>[r] \ar[d]^{\beta_2}& 
			\tX \ar[r]^(.3){\pi_X}\ar[dr]^(.3){\pi_I} \ar[d]^{\beta_1} & X \\
	X^3\times I \ar@<1ex>[r]^{\pi_{L,R,C}}\ar[r]\ar@<-1ex>[r] & X^2\times I \ar@<0.6ex>[r]^{\pi_{L,R}}\ar@<-0.6ex>[r] & X\times I \ar[ur] \ar[r] 
			& I
}
\end{equation}
\end{thm}	

\begin{rmk} 
We `overload' the notation for projections, so that for instance, $\pi_I :
\tX^3 \to I$ should mean the unique b-fibration lifting $X^3 \times I \to I$,
which by the above theorem is given by any appropriate composition.
\end{rmk}

The following identifications of various submanifolds and boundary faces of
$\tX^2$ are fundamental to the calculus, and are easily verified in local
coordinates.
\[
\begin{aligned}
	\fsc &\cong \ol{\scT_{\pa X} X} \times I \to \pa X \times I \\
	\ftf &\cong (\pa X\times [0,1])^2_\mathrm{b,sc}, \\
	\fzf &\cong X^2_\mathrm{b}, \\
	\pi^{-1}_I(\tau) &\cong X^2_\mathrm{sc}, \quad \tau > 0, \\
	\Delta &\cong \tX,
\end{aligned}
\]
where $\Delta \subset \tX^2$ denotes the lifted fiber diagonal.

\subsection{Densities} We will make use primarily b half densities for operator
kernels and functions, in order to facilitate the invocation of the pushforward
theorem for polyhomogeneous conormal distributions. Observe on the unresolved
spaces $X^n\times I$, $n = 1,2,3$ there are canonical identifications
\begin{equation}
\begin{aligned}
	\bOh(X\times I) &\cong \pi^\ast_X(\bOh(X))\otimes \pi^\ast_I(\bOh(I))\\
	\bOh(X^2\times I) &\cong \pi_{X,L}^\ast(\bOh(X))\otimes \pi_{X,R}^\ast(\bOh(X))\otimes\pi_I^\ast(\bOh(I))\\
	\bOh(X^3\times I) &\cong \pi_{X,L}^\ast(\bOh(X))\otimes \pi_{X,C}^\ast(\bOh(X))\otimes\pi_{X,R}^\ast(\bOh(X))\otimes\pi_I^\ast(\bOh(I))
\end{aligned}
	\label{E:prod_densities}
\end{equation}
where $\pi_{X,R}$, etc. are shorthand for $\pi_X \circ \pi_R$ and so on.

\begin{lem} \label{L:tra:densities}
On $\tX^2$ and $\tX^3$, respectively, there are canonical identifications
\[
\begin{aligned}
	&\pi^\ast_R(\bOh(\tX))\otimes \pi^\ast_L(\bOh(\tX)) \cong \rho_\fsc^{n/2} \bOh(\tX^2) \otimes \pi^\ast_I(\bOh(I)), \quad \text{and} \\
	&\pi^\ast_R(\bOh(\tX^2))\otimes\pi^\ast_C(\bOh(\tX^2))\otimes\pi^\ast_L(\bOh(\tX^2)) \cong (\sigma_\fsc^{n/2}\bO(\tX^3)\otimes \pi_I^\ast(\bO^2(I))
\end{aligned}
\]
where $\sigma_\fsc$ is a product of all the boundary defining functions
$\rho_{G,L/R/C}$ and $\rho_{J,L/R/C}$ for the faces obtained by blowing up
$G_{R/L/C}$ and $J_{R/L/C}$ in the process of obtaining $\tX^3$.
\end{lem}

\begin{proof}
The proof follows from \eqref{E:prod_densities} and the standard fact that
\[
	\beta^\ast(\Oh(X)) = \rho_\mathrm{ff}^{(\codim(Y)-1)/2}\Oh([X; Y])
\]
with respect to a blow-down map $\beta : [X; Y] \to X$, where
$\rho_\mathrm{ff}$ denotes a boundary defining function for the front face of
the blow-up.
\end{proof}

We will make use of the canonical trivializing section $\nu^s :=
\abs{\frac{d\tau}{\tau}}^s \in C^\infty(I; \Omega^s)$ and its pullback to
various spaces.

We denote the {\em kernel density bundle} $\kd \to \tX^2$ by
\[
	\kd = \rho_\fsc^{-n/2}\bOh(\tX^2).
\]
This convention normalizes the densities so that the kernel of the identity
operator on b half densities has smooth asymptotic expansion of order 0 at all
boundary faces meeting the lifted diagonal.

\begin{lem} \label{L:tra:density_bundle_restriction}
The restriction of the kernel density bundle gives the following identifications
\begin{align*}
	(\kd)\rst_{\fzf} &\cong \bOh(X^2_b) \\
	(\kd)\rst_{\ftf} &\cong \rho_\fsc^{-n/2}\bOh\bpns{(\pa X\times [0,1])_\mathrm{b,sc}^2} \\
	(\kd)\rst_{\pi_I^{-1}(\tau)} &\cong \rho_\fsc^{-n/2} \bOh(X^2_\fsc) \cong \scOh(X^2_\fsc), \quad \tau > 0 \\
	(\kd)\rst_{\Delta} &\cong \rho_\fsc^{-n}\bO((\tX)_\mathrm{fib})\otimes\bOh(I)
\end{align*}
where $(\tX)_\mathrm{fib}$ denotes the (generalized) fiber of the b-fibration $\pi_I : \tX \to I$.
\end{lem}
\begin{proof}
The restriction of b half densities to boundary faces is well-defined,
corresponding locally to the cancellation of a boundary defining factor, from
which the first and second identifications follow.

Along the submanifold $\pi_I^{-1}(\tau), \tau > 0$,
\[
	\bOh(\tX^2)\rst_{\pi_I^{-1}(\tau)} \cong \bOh(\pi_I^{-1}(\tau))\otimes \bOh(I) \cong \bOh(\pi_I^{-1}(\tau))
\]
using the trivializing section $\nu^{1/2}$ of $\bOh(I)$.

For the last claim, we note that, near $\fzf\cap \Delta$, 
\begin{align*}
	\bOh(\tX^2) &\cong \bOh(X)\otimes \bOh(N\Delta)\otimes \bOh(I) \\
	&\cong \bO(X)\otimes \bOh(I).
\end{align*}
Similarly, near $\fsc \cap \Delta$, we have 
\begin{align*}
	\kd &= \rho_\fsc^{-n/2}\bOh(X)\otimes \bOh(N\Delta) \otimes \bOh(I) \\
	&= \rho_{\fsc}^{-n} \bO(X) \otimes \bOh(I). \qedhere
\end{align*}
\end{proof}

\subsection{The Calculus}

Fix a vector bundle $V \to X$, and denote also by $V \to \tX^i$ $i \in
\set{1,2,3}$ the pullback of $V$ to the single, double and triple spaces. 
The {\em b-sc transition pseudodifferential operators} are defined by
\begin{align*}
	\tP^{m,\cE}(\tX; V\otimes \bOh) &= \Aphg^\cE I^m(\tX^2, \Delta; \End(V)\otimes \kd), \\
		\cE &= (E_\fzf, E_{\ftf}, E_{\fsc}, E_{\flb_0}, E_{\frb_0}, \infty_{\flb}, \infty_{\frb}, \infty_{\fbf})
\end{align*}
where $I^m(\tX^2, \Delta)$ denotes the space of distributions conormal to
$\Delta$ in the sense of H\"ormander, with symbol order $m$. In particular,
these distributions vanish to infinite order at $\flb, \frb$, and $\fbf$.  For
notational convenience we are denoting the bundle $\Hom(\pi_R^\ast V, \pi_L^\ast V)
\to \tX^2$ simply as $\End(V)$.

A distinguished subclass of these operators form the {\em small calculus}
\begin{align*}
	\tP^{m, (e_\fzf, e_{\ftf}, e_{\fsc})}(\tX; V\otimes \bOh) &= \tP^{m, \cE}(\tX; V\otimes \bOh), \\ \text{with } 
		\cE &= (e_\fzf, e_{\ftf}, e_{\fsc}, \infty_{\flb_0}, \infty_{\frb_0}, \infty_{\flb}, \infty_{\frb}, \infty_{\fbf})
\end{align*}
where $e_i \in \Z$ are identified with the smooth index sets $\set{(e_i + \N, 0)}$.

The action of $P \in \tP^{m,\cE}(\tX; V\otimes \bOh)$ on $u \in
\Cdot^\infty(\tX; V\otimes \bOh)$ (here $\dot C^\infty$ denotes smooth sections
vanishing to infinite order at all boundary faces) is given by
\[
	P u = (\pi_L)_\ast \pns{\kappa_P\cdot\pi_R^\ast (u)\,\pi_I^\ast (\nu^{-1/2})},
\]
where $\kappa_P$ is the Schwartz kernel of $P$.

\begin{prop}
The action of $\tP^\ast$ on $\dot C^\infty$ extends to an operation
\[
	\tP^{m,\cE}(\tX; V\otimes \bOh) \cdot \Aphg^\cF(\tX; V\otimes\bOh) \subset \Aphg^\cG(\tX; V\otimes \bOh)
\]
where
\begin{align*}
	G_\fsc &= F_\fsc + E_\fsc \\
	G_\ftf &= \pns{F_\ftf + E_{\ftf} }\overline{\cup} \pns{F_\fzf + E_{\frb_0}}\\
	G_\fzf &= \pns{F_\fzf + E_{\fzf}} \overline{\cup} \pns{F_\ftf + E_{\flb_0}}
\end{align*}
\end{prop}
\begin{rmk}
In particular, the small calculus maps $\Aphg^\cF$ to $\Aphg^\cG$ with 
\begin{align*}
	G_\fsc &= F_\fsc + e_\fsc \\
	G_\ftf &= F_\ftf + e_{\ftf} \\
	G_\fzf &= F_\fzf + e_{\fzf}
\end{align*}
\end{rmk}
\begin{proof}
This is a direct consequence of Theorems \ref{T:man:pullback} and
\ref{T:man:pushforward}.  Indeed, taking $u \in \Aphg^{\cF}(\tX; V\otimes
\bOh)$, it follows that
\[
	\kappa_P\cdot\pi_R^\ast(u)\,\pi_I^\ast(\nu^{-1/2}) \in 
	\Aphg^\cH I^m(\tX^2, \Delta; V\otimes \rho_\fsc^{-n/2}\bOh(\tX^2)\otimes \pi_R^\ast\bOh(\tX)\otimes \pi_I^\ast(\bO^{-1/2}(I)))
\]
where
\begin{align*}
	H_\fsc &= E_\fsc + F_\fsc &
	H_\fbf &= E_\fbf + F_\fsc = \infty + F_\fsc = \infty \\
	H_\flb &= E_\flb + F_\fsc = \infty + F_\fsc = \infty &
	H_\ftf &= E_\ftf + F_\ftf \\
	H_\flbz &= E_\flbz + F_\ftf &
	H_\frbz &= E_\frbz + F_\fzf \\
	H_\fzf &= E_\fzf + F_\fzf.
\end{align*}
The index sets $G_i$ are obtained from the pushforward theorem, since all boundary
exponents of $\pi_L$ are either 0 or 1 and 
\begin{align*}
	\pi_L^{-1}(\fsc) &= \fbf \cup \fsc \cup \frb \\
	\pi_L^{-1}(\ftf) &= \ftf \cup \frbz \\
	\pi_L^{-1}(\fzf) &= \flbz \cup \fzf.
\end{align*}
The interior conormal singularity is killed since $\pi_L$ is transversal to
$\Delta$.  

It remains to verify that the density bundles behave as expected.  The claim is
that the pushforward under $\pi_L$ of $\kappa_P\, \pi_R^\ast(u)
\pi_I^\ast(\nu^{-1/2})$ can be identified with a section of $\bOh(\tX)\otimes
V$, or equivalently that it pairs with $\bOh(\tX)$ to produce an element
of $\bO(\tX)\otimes V$.  Thus let $\gamma \in C^\infty(\tX; \bOh(\tX))$ and
consider
\[
	\gamma \, (\pi_L)_\ast\pns{\kappa_P\cdot\pi_R^\ast(u)\pi_I^\ast(\nu^{-1/2})} 
	= (\pi_L)_\ast\pns{\pi_L^\ast(\gamma)\,\kappa_P\cdot\pi_R^\ast(u)\pi_I^\ast(\nu^{-1/2})}.
\]
The element in parentheses on the right hand side is a section of 
\[
	\pi_L^\ast(\bOh(\tX))\otimes \pi_R^\ast(\bOh(\tX))\otimes\rho_\fsc^{-n/2} \bOh(\tX^2)\otimes \pi_I^\ast(\bO^{-1/2}(I))\otimes V,
\]
By Lemma \ref{L:tra:densities}, we can identify this with the bundle
\[
	\rho_\fsc^{n/2}\bOh(\tX^2)\otimes \pi_I^\ast(\bO^{1/2}(I))\otimes \rho_\fsc^{-n/2} \bOh(\tX^2)\otimes \pi_I^\ast(\bO^{-1/2}(I))\otimes V
	\cong \bO(\tX^2) \otimes V,
\]
and $(\pi_L)_\ast$ maps this into $\bO(\tX)\otimes V$ as claimed. 
\end{proof}

Composition of the operators is defined by pulling back to the triple space,
multiplying and pushing forward:
\[
	\kappa_{P\circ Q} = (\pi_C)_\ast\pns{\pi_L^\ast(\kappa_P)\cdot \pi_R^\ast(\kappa_Q)\,\pi_I^\ast(\nu^{-2})}
\]

\begin{prop}
The composition of operators is well-defined, with
\[
	\tP^{m,\cE}(\tX; V\otimes \bOh) \circ \tP^{m', \cF}(\tX; V\otimes \bOh) \subset \tP^{m+m', \cG}(\tX; V\otimes \bOh)
\]
where
\begin{align*}
	G_\fsc &= E_\fsc + F_\fsc \\
	G_\fzf &= \pns{E_\fzf + F_\fzf} \overline{\cup} \pns{E_{\frb_0} + F_{\flb_0}} \\
	G_{\ftf} &= \pns{E_{\ftf} + F_{\ftf}} \overline{\cup} \pns{E_{\flb_0} + F_{\frb_0}} \\
	G_{\flb_0} &= \pns{E_{\flb_0} + F_{\fzf}} \overline{\cup} \pns{E_{\fbf_0} + F_{\flb_0}} \\
	G_{\frb_0} &= \pns{E_{\fzf} + F_{\frb_0}} \overline{\cup} \pns{E_{\frb_0} + F_{\fbf_0}}
\end{align*}
\label{P:tra:composition}
\end{prop}
\begin{rmk}
In particular, the small calculus composes as
\[
	\tP^{m,(e_\fzf, e_{\ftf}, e_{\fsc})} \circ \tP^{m', (f_{\fzf}, f_{\ftf}, f_{\fsc})} 
		\subset \tP^{m+m', (e_\fzf + f_\fzf, e_{\ftf} + f_{\ftf}, e_\fsc+ f_\fsc)}.
\]
\end{rmk}
\begin{proof}
First we consider how the densities behave. From Lemma~\ref{L:tra:densities} 
it follows that $\pi_L^\ast (\kappa_P)\pi_R^\ast(\kappa_Q)\pi_I^\ast(\nu^{-2})$
is a section of $\pns{\rho_{J,C}\rho_{G,C}}^{-n/2}\bO(\tX^3)\otimes \pi_C^\ast\bO^{-1/2}(\tX^2)$.
This may subsequently be identified with $\Omega_\mathrm{fib}(\tX^3)\otimes \pi_C^\ast(\kd)$
where $\Omega_\mathrm{fib}$ denotes fiber densities with respect to $\pi_C$.

The interior conormal singularities of $\kappa_P$ and $\kappa_Q$ compose
transversally as in the classical case, and the only contribution to survive
the pushforward comes from $\pi_C^{-1}(\Delta)$, since the conormal singularity
everywhere else is transversal to $\pi_C$.

Finally, the index sets are determined by the pushforward theorem, once we
identify the relationships between the inverse images under $\pi_{L/R/C}$ of
the boundary hypersurfaces of $\tX^2$, in $\tX^3$.  In fact, as all boundary
faces in question are the lifts of boundary faces of the product spaces
$X^2\times I$ and $X^3\times I$ under blowups, it suffices to consider the maps
$\pi_{L/R/C} : X^3 \times I \to X^2 \times I$, using commutativity of
(\ref{E:tra:b_fibns}).  

For instance, $\ftf \subset \tX^2$ is the lift under $\beta$ of the face $\pa X
\times \pa X \times \set{0} \subset X^2\times I$.  So consider $\pi^{-1}_C(\pa
X\times\pa X \times \set{0}) \subset X^3\times I$.  This consists of two
boundary faces\footnote{Though one face is included in the other in $X^3 \times
I$, this inclusion relationship is not preserved under blowup to $\tX^3$, so we
consider them separately.}, namely
\[
	\pa X \times \pa X \times \pa X \times \set{0} \quad \text{and}\quad \pa X\times X \times \pa X \times \set{0}.
\]
The first face projects down to $\pa X\times\pa X \times \set{0}$ under both
$\pi_L$ and $\pi_R$, corresponding to the face $\ftf \subset \tX^2$, while the
second face projects to $\pa X \times X \times \set{0}$ under $\pi_L$
(corresponding to $\flbz \subset \tX^2$) and to $X\times \pa X \times \set{0}$
under $\pi_R$ (corresponding to $\frbz \subset \tX^2$).  From this we conclude
that
\[
	G_\ftf = \pns{E_\ftf + F_\ftf} \overline\cup \pns{E_\flbz + F_\frbz}.
\]
The index sets for the other faces are obtained similarly.
\end{proof}

\subsection{Normal operators and symbols} 

The leading order term in the asymptotic expansion of operator kernels at the
boundary faces intersecting the lifted diagonal play a special role.  

\begin{defn}
Given $P \in \tP^{m,\cE}(\tX; V\otimes \bOh)$ with kernel $\kappa_P$, let
\begin{itemize}
\item $\Nt(P)$ be the restriction of $\kappa_P$ to $\ftf$,
\item $\ssym(P)$ be the fiberwise Fourier transform of the restriction of $\kappa_P$ to $\fsc$ with respect to the vector bundle structure on $\fsc$.
\item $N_\tau(P)$ be the restriction of $\kappa_P$ to $\pi_I^{-1}(\tau), \tau > 0$, and
\item $N_\fzf(P)$ be the restriction of $\kappa_P$ to $\fzf$.
\end{itemize}
\end{defn}
Here restriction means the restriction of the leading order term in the
asymptotic expansion as a section of the kernel density bundle, and it follows
that the distributions $N_\ast(P)$ are well-defined by transversality of the
various faces and the lifted diagonal. The composition theorem allows these
distributions to be interpreted as various model operators.
\begin{prop}
With $P \in \tP^{m,\cE}(\tX; V\otimes \bOh), Q \in \tP^{m', \cF}(\tX; V\otimes \bOh)$, there are identifications
\begin{align*}
	\ssym(P) &\in C^\infty(\scT^\ast_{\pa X} X\times I; \End(V)\otimes \Oh)) \\
	\Nt(P) &\in \Psi_{\mathrm{b,sc}}^{m, (E_\fzf, E_{\flb_0}, E_{\frb_0}), E_\fsc}(\pa X \times {}_\mathrm{b}[0,1]_\mathrm{sc}; V\otimes \bOh) \\
	N_\fzf(P) &\in \bP^{m,(E_{\fbf_0}, E_{\flb_0}, E_{\frb_0})}(X; V\otimes \bOh) \\
	N_\tau(P) &\in \scP^{m,E_\fsc}(X; V\otimes \bOh)\quad \tau > 0.
\end{align*}
With respect to composition, these satisfy
\begin{align*}
	\ssym(P\circ Q) &= \ssym(P)\ssym(Q) \\
	\Nt(P\circ Q) &= \Nt(P)\circ \Nt(Q) \\
	N_\fzf(P\circ Q) &= N_\fzf(P)\circ N_\fzf(Q) \\
	N_\tau(P\circ Q) &= N_\tau(P)\circ N_\tau(Q) 
\end{align*}
provided $\re(E_{\frb_0} + F_{\flb_0}) > \re(E_\fzf + F_\fzf)$ in the case of
$N_\fzf(P\circ Q)$, and provided $\re(E_{\flb_0} + F_{\frb_0}) > \re(E_{\fbf_0}
+ F_{\fbf_0})$ in the case of $\Nt(P\circ Q)$.
\label{P:tra:normal_op_comp}
\end{prop}

\begin{rmk}
$N_{\fbf_0}(P)$ is an operator on $\pa X \times {}_\mathrm{b}[0,1]_\mathrm{sc}$
which is of b type near 0 and scattering type near 1.  It has index sets
$(E_{\fzf}, E_{\flb_0}, E_{\frb_0})$ as a b operator, and the index set
$E_\fsc$ as a scattering operator.  $\ssym(P)$ is just a family of scattering
symbols, parameterized smoothly by $\tau/x \in [0,1)$.
\end{rmk}
\begin{proof}
Observe that in the triple space $\pi^{-1}_I(\tau) \cong X^3_\mathrm{sc}$ for
$\tau > 0$, so for fixed $\tau > 0$ composition coincides with the composition
of scattering operators.  From this the composition formulas for $\ssym$ and
$N_\tau$ follow.

Next consider $N_\fzf(P)$.  From the identification $\fzf \cong
X^2_\mathrm{b}$, this may be viewed as the Schwartz kernel of a b operator on
$X$, it must be shown that $N_\fzf(P\circ Q) = N_\fzf(P)\circ N_\fzf(Q)$ under
appropriate conditions.  As $\fzf \subset \tX^2$ is the lift of $X \times X
\times \set{0} \subset X^2 \times I$, consider the inverse image of the this
boundary face under $\pi_C$ in $X^3\times I$:  
\[
	\pi_C^{-1}(X \times X\times \set{0}) = X\times X \times X\times \set{0} \cup X\times \pa X\times X\times \set{0}.
\]

The first face projects to $X^2 \times \set{0} \subset X^2\times I$ under both
$\pi_L$ and $\pi_R$, while the second projects to $X\times \pa X\times \set{0}$
under $\pi_L$ and to $\pa X\times X\times \set{0}$ under $\pi_R$. (This is the
reason for the index set behavior $G_\fzf = \pns{E_\fzf + F_\fzf}
\overline{\cup} \pns{E_{\frb_0} + F_{\flb_0}}$ with respect to composition.)
Provided $\re(E_\frbz + F_\flbz) > \re(E_\fzf + F_\fzf)$, the contribution of
the former vanishes upon restriction to $\fzf$, so under this condition the
value $(\kappa_{P\circ Q})\rst_{\fzf}$ is determined by the composition on the
lift of $X\times X \times X \times\set{0}$ in $\tX^3$.  

A straightforward coordinate computation shows that the lift of this face is
isomorphic to the b triple space $X^3_\mathrm{b}$, with $\pi_{L/R/C}$ restricting
to the analogous b-fibrations $X^3_\mathrm{b} \to X^2_\mathrm{b}$, and
$N_\fzf(P\circ Q) = N_\fzf(P)\circ N_\fzf(Q)$ follows at once.

Similar considerations apply to $\Nt(P)$. From the composition formula,
provided $\re(E_\flbz + F_\frbz) > \re(E_\ftf + F_\ftf)$, the contribution to
$(\kappa_{P\circ Q})\rst_{\ftf}$ consists of the composition of $P$ and $Q$ on
the boundary hypersurface in $\tX^3$ which is the lift of $\pa X\times \pa
X\times \pa X\times \set{0}$ in $X^3\times I$. A coordinate computation again
shows that this face is diffeomorphic to the triple space $(\pa X\times
[0,1])^3_{\mathrm{b,sc}}$ and that $\pi_{L/R/C}$ induce the appropriate
b-fibrations $(\pa X \times [0,1])^2_\mathrm{b,sc}.$
\end{proof}

\begin{prop}
The symbols and normal operators are related by
\begin{equation}
\begin{aligned}
	\ssym(P)\rst_{\pi_I^{-1}(\tau)} & = \ssym(N_\tau(P)) \\
	\ssym(P)\rst_{\fsc\cap \ftf} &= \ssym(\Nt(P)) \\
	I(\Nt(P),\lambda) &= I(N_\fzf(P),-\lambda)
\end{aligned}
	\label{E:calc_normal_relns}
\end{equation}
\label{P:calc_normal_relns}
\end{prop}
\begin{proof}
The first two equations are a straightforward consequence of the geometry, the
symbols in question being given by the fiberwise Fourier transform along
$\pi_I^{-1}(\tau) \cap \fsc$ and $\fsc \cap \ftf$, respectively, using the
vector bundle structure induced from the identification $\fsc \cong \ol
{\scT_{\pa X} X}\times I \to \pa X\times I.$

The third equation follows similarly from the geometry, the indicial operators
of $\Nt(P)$ and $N_\fzf(P)$ being given by the fiberwise Mellin transform of
the restriction of $\kappa_P$ to $\ftf \cap \fzf \cong [0,\infty] \times \pa
X^2.$ The difference in the sign of $\lambda$ is explained by the fact that the
fiber variable $s \in [0,\infty]$ (with respect to which the Mellin transform
is computed) differs by $s \mapsto 1/s$ between the two operators.

To see this, suppose $(x,y,\tau)$ and $(x',y',\tau)$ are two copies of
coordinates on $X\times I$ near $\pa X \times \set{0},$ so that
$(x,y,x',y',\tau)$ form coordinates on $X^2 \times I$ near the corner. After
blow-up, coordinates near $\fzf \cong \bX^2$ are given by $(x,y,s =
x'/x,y',\tau)$ (where $\fzf = \set{\tau = 0}$) and the action of the indicial
operator $I(N_\fzf(P))$ is given by integration in $y'$ and (multiplicative)
convolution in $s \in [0,\infty)$ with fiber density $ds/s.$ (This follows for
instance from the fact that $\beta^\ast\pns{x'\pa_{x'}}\rst_{\ftf \cap \fzf} =
s\pa_s.$) On the other hand, coordinates for $\ftf \cong (\pa X \times
[0,1])^2_\mathrm{b,sc}$ near the b front face $\ftf \cap \fzf$ are given by 
\[
	(t, y, s', y') = (\tau/x, y, t'/t = (\tau/x')(x/\tau) = 1/s, y')
\]
and $I(N_\ftf(P))$ acts by convolution in $s' = 1/s$ with fiber density $ds'/s'.$
\end{proof}

\subsection{Symbols of Differential operators}\label{S:calculus_diffl}

Next we discuss the Lie algebra giving rise to the transition calculus, and
give an alternate definition of the normal operators for smooth differential
operators. The extension to $\cB\tDiff^\ast(\tX; V)$ is covered by the
pseudodifferential calculus already discussed.

Consider the b fibration $\pi_I : \tX \to I$.  Let $\cV_\mathrm{b,f}(\tX)
\subset \bV(\tX)$ denote the subset of the b vector fields on $\tX$ (so those
tangent to the boundary hypersurfaces), which are tangent to the fibers of
$\pi_I$.  Tangency is preserved under Lie bracket, so $\cV_\mathrm{b,f}(\tX)$
is a Lie subalgebra. Then the {\em b-sc transition vector fields} are defined to
be those vector fields which additionally vanish at $\fsc$:
\[
	\cV_t(\tX) = \rho_\fsc \cV_\mathrm{b,f}(\tX).
\]

\begin{lem}
\mbox{}
\begin{enumerate}
[{\normalfont (a)}]
\item \label{I:calc_vfield_one} $[\cV_t(\tX),\cV_t(\tX)] \subset \rho_\fsc \cV_t(\tX)$, so $\cV_t(\tX)$ is a
well-defined Lie subalgebra.
\item \label{I:calc_vfield_two} Evaluation of $\cV_t(\tX)$ at $\fsc$ takes values
in the abelian Lie algebra $\cV_t(\tX)/\rho_\fsc \cV_t(\tX).$
\item \label{I:calc_vfield_three} Evaluation at $\ftf$ takes values in $\cV_t(\tX)/\rho_\ftf
\cV_t(\tX) \cong \cV_{\mathrm{b,sc}}(\pa X \times \bIsc).$ 
\item \label{I:calc_vfield_four} Evaluation at $\fzf$ takes values in
$\cV_t(\tX)/\rho_\fzf \cV_t(\tX) \cong \bV(X)$.
\end{enumerate}
\end{lem}

\begin{proof}
\eqref{I:calc_vfield_one} follows exactly as in the proof of the
identity $[\scV(X), \scV(X)] \subset x\scV(X)$, using the fact that
$\cV_\mathrm{b,f}(\tX) \cdot \rho_\fsc C^\infty(\tX) \subset \rho_\fsc
C^\infty(\tX)$. Then \eqref{I:calc_vfield_two} follows from \eqref{I:calc_vfield_one} and
the fact that $\rho_\fsc \cV_t(\tX)$ is an ideal.

\eqref{I:calc_vfield_three} follows from the fact that, since $\ftf$ is a (generalized)
fiber of $\pi_I$, $\cV_t(\tX)/\rho_\ftf \cV_t(\tX)$ consists of vector fields on $\ftf$ which
are tangent to its boundaries $\ftf \cap \fsc$ and $\ftf \cap \fzf$, with an additional vanishing
factor at $\ftf\cap \fsc.$ A similar consideration results in \eqref{I:calc_vfield_four}.
\end{proof} 

The {\em transition differential operators} are the enveloping algebra of
$\cV_t(\tX)$; equivalently, they are defined by iterated composition of
$\cV_t(\tX)$ as operators on $C^\infty(\tX; V)$.
\[
	\tDiff^k(\tX; V) = \set{\sum_{j\leq k} a_j\,V_1\cdots V_j\;;\; V_i \in \cV_t(\tX), a_j \in C^\infty(\tX; \End(V))}
\]

The evaluation maps above then extend to normal operator homomorphisms
\begin{align*}
	N_\fsc &: \tDiff^k(\tX; V) \to \Diff^k_{\mathrm{I,fib}}((\scT_{\pa X} X)\times I; V) \\
	\Nt &: \tDiff^k(\tX; V) \to \Diff^k_\mathrm{b,sc}(\pa X \times I; V) \\
	N_\tau &: \tDiff^k(\tX; V) \to \scDiff^k(X; V) \\
	N_\fzf &: \tDiff^k(\tX; V) \to \bDiff^k(X; V),
\end{align*}
where $\Diff^k_\mathrm{I,fib}((\scT_{\pa X} X)\times I; V)$ denotes translation
invariant differential operators along the fibers of $\scT_{\pa X} X\times I
\to \pa X\times I$ which are smoothly parameterized by the base.  Fiberwise
Fourier transform of the first of these gives the scattering symbol
\[
	\ssym : \tDiff^k(\tX; V) \to C^\infty(\scT^\ast_{\pa X} X \times I; \End(V)).
\]

\begin{prop}
For any $P \in \tDiff^k(\tX; V)$, the indicial families of $\Nt(P)$ and $N_\fzf(P)$ are related by
\[
	I(\Nt(P),\lambda) = I(N_\fzf(P),-\lambda).
\]
\label{P:trans_diffl_indfam}
\end{prop}
\begin{proof}
It suffices to verify this in the case of a vector field $W \in \cV_t(\tX)$. By
definition $W$ is tangent to the fibers of $\pi_I : \tX \to I,$ which has the
local coordinate form
\[
	\pi_I : (x,y,t) \to \tau = xt
\]
near $\fzf \cap \ftf$ where $t = \tau/x$ is a blow-up coordinate, and $y \in
\pa X.$ In these coordinates, $\cV_t(\tX)$ is span over $C^\infty(\tX)$ of the
vector fields $\pa_{y_i}$ and $x\pa_x - t\pa_t$, so
\[
	W \loceq a_0(x,y,t)(x\pa_x - t\pa_t) + \sum_i a_i(x,y,t)\,\pa_{y_i} 
\]
from which it follows that
\[
\begin{aligned}
	I(N_\fzf(W),\lambda) &= a_0(0,y,0) (\lambda) + \textstyle\sum_i a_i(0,y,0)\,\pa_{y_i} \\
	I(N_\ftf(W),\lambda) &= a_0(0,y,0) (- \lambda) + \textstyle\sum_i a_i(0,y,0)\,\pa_{y_i}. \qedhere
\end{aligned}
\]
\end{proof}

By lifting vector fields to $\tX^2$ from the left or the right and applying
them to the kernel of the identity, it follows from the analogous results for
the b and scattering calculi that these definitions of $N_\ast(P)$ and
$\ssym(P)$ agree with those defined for pseudodifferential operators on
$\tX^2.$

\subsection{The Trace}  \label{S:calculus_trace}

For operators in $\bP^\ast(X)$ or $\scP^\ast(X)$ which extend to be trace
class, the trace may be computed by restricting to the diagonal and integrating
(taking the trace fiberwise if operating on sections of a vector bundle).  Here
sufficient conditions for $A \in \bP^{m,\cE}(X; V)$ to be trace class are that
$m < -\dim(X)$ and $\re E_{\fff} > 0,$ and for $B \in \scP^{m,e_{\fsc}}(X; V)$
that $m < -\dim(X)$ and $e_\fsc > n.$ 

We define an analogous `families trace' for sufficiently well-behaved $P \in
\tP^\ast(\tX; V)$.
\begin{defn}
Let $P \in \tP^{m,\cE}(\tX; V)$ with $m < -\dim(X)$, $\re\,E_\fsc > n$. Then
$\Tr(P) \in \cAphg^F(I)$ is defined by
\[
	\Tr(P) = (\pi_I)_\ast\bpns{\tr \pns{\kappa_P \rst_{\Delta}}}\nu^{-1/2} \in \cAphg^F(I)
\]
where $F = E_\ftf \ol \cup E_{\fzf},$ $\tr(\cdot)$ denotes the trace in the
endomorphism bundle $\End(V),$ $\pi_I : \tX \to I$ is the lift of the
projection $X\times I \to I,$ and $\nu^s = \abs{\frac{d\tau}{\tau}}^s$ is the canonical
trivializing section of $\bO^s(I).$
\label{D:trans_trace}
\end{defn}

To see that this operation is well-defined, observe that $m < -\dim(X)$ implies
that $\kappa_P$ is continuous so that the restriction to the diagonal is
well-defined, and by Lemma~\ref{L:tra:densities} results in a section $\kappa_P
\rst_{\Delta} \in \cA^\cF\bpns{\tX; \rho_\fsc^{-n}{\bO}_\mathrm{,fib}\otimes
\bOh(I)}$, where ${\bO}_\mathrm{,fib}$ denotes fiber b-densities with respect to
$\pi_I : \tX \to I.$ The condition $\re\,E_\fsc > n$ implies that, viewed as a
(fiber) b-density, $\kappa_P \rst_{\Delta}$ has index set at $\fsc$ with
strictly positive real part so that the pushforward by $\pi_I$ exists,
giving a section of $\cAphg^F(I; \bOh(I))$ after which the density factor
is removed by multiplication by $\nu^{-1/2}.$ The index set $F = E_\ftf
\overline \cup E_\fzf$ is a consequence of Theorem~\ref{T:man:pushforward}.

It is clear that for $\tau > 0$, the value of $\Tr(P)$ at $\tau$ coincides with 
the trace of $N_\tau(P)$ as a scattering operator:
\[
	\Tr(P)(\tau) = \Tr(N_\tau(P)) 
\]
More subtle is the fact that $\Tr(P)(0)$ coincides with $\Tr(N_\ftf(P)) +
\Tr(N_\fzf(P))$, provided these exist.

\begin{prop}
Provided $P \in \tP^{m,\cE}(\tX; V)$ satisfies $m < - \dim(X)$, $\re\,E_\fsc > \dim(X)$,
$\re E_\ftf = \re E_\fzf = 0$, with $N_\fzf(P) \in \bscP^\ast$ and $N_\ftf(P) \in
\bP^\ast$ trace class, then $\Tr(P)$ is a continuous function of $\tau$ down to
$\tau = 0$ and
\begin{equation}
	\Tr(P)(0) = \Tr(N_\ftf(P)) + \Tr(N_\fzf(P))
	\label{E:trace_zero}
\end{equation}
\label{P:trace_at_zero}
\end{prop}
\begin{proof}
According to the pushforward theorem, $\Tr(P) \in \cAphg^F(I)$ where $F =
E_\fzf \overline \cup E_\ftf$. In particular the leading order index is $(0,0)
\overline \cup (0,0) = (0,1)$, meaning a priori $\Tr(P)(\tau)$ has leading
order asymptotic $\log(\tau)$ at $\tau = 0.$ To prove the first part it
suffices to show this leading order behavior does not occur, since the next
order term is $\cO(\tau^0) = \cO(1)$ (corresponding to index $(0,0)$), implying
boundedness and hence continuity.

We claim that the offending log term does not occur provided that the
restriction $(\kappa_P \rst_{\Delta}) \rst_{\ftf \cap \fzf}$ to the corner
$\ftf \cap \fzf$ (which is canonically a density on $\ftf \cap \fzf \cong \pa
X$) integrates to 0. From this the result follows, for then the assumption that
$N_\fzf(P)$ and $N_\ftf(P)$ are trace class means they have vanishing Schwartz
kernels at $\ftf \cap \fzf \subset \tX^2$, so in particular $(\kappa_P
\rst_\Delta) \rst_{\ftf \cap \fzf}$ vanishes identically.

Thus suppose $u \in \cAphg^\cF(\tX; \bO)$ is a bounded continuous density with
leading order index $(0,0) \in F_\fzf,\  F_\ftf$, and that $u \rst_{\ftf\cap \fzf}
\in C^\infty(\pa X; \Omega)$ has vanishing integral. The b-fibration $\tX \to I$
is given locally near $\ftf \cap \fzf$ by 
\[
	\pa X \times [0,1)^2 \ni (y,t,x) \mapsto \tau = tx \in I.
\]
using coordinates $y \in \pa X$, the boundary defining function $x$ on $X$, and
blow-up coordinate $t = \tau/x.$ Let $\phi \in C_c^\infty([0,1)^2)$
be a smooth cutoff which is identically equal to 1 on a small neighborhood of
$\set{0,0}.$

Recall that polyhomogeneity of $\phi u \in \cAphg^\cF(\pa X \times[0,1)^2,
\bO)$ is equivalent to the multi-Mellin transform\footnote{Observe that the
density element $dt/t\,dx/x$ is part of $u$.}
\[
	\cM_2(\phi\,u)(\eta,\xi) = \int_{[0,1)^2} t^{-\eta}x^{- \xi} \phi(t,x)\,u(y,t,x)\,\frac{dt}{t}\,\frac{dx}{x}
\]
being a product-type meromorphic function on $\C^2$ valued in $C^\infty(\pa X;
\Omega)$ with poles of order $p + 1$ and $q + 1$ along $\set{\eta =  z_1}$
and $\set{\xi =  z_2}$ respectively, for each $(z_1,p) \in F_{\fzf}$ and
$(z_2,q) \in F_\ftf$, along with a uniform decay condition in strips $\abs{\re
\xi} \leq C$ and $\abs{\re \eta} \leq C'$. 

Similarly, the index set $F$ of $v \in \cAphg^F(I;\bO)$ is determined by the
orders and locations of the poles of 
\[
	\cM_1(\phi v)(\lambda) = \int_{I} \tau^{- \lambda} \phi(\tau)v(\tau)\,\frac{d\tau}{\tau}
\]
where a pole of order $p + 1$ at $\lambda = z$ corresponds to $(z,p) \in F.$

Thus to show $f_\ast \phi\, u$ has no log term corresponding to $(0,1) \in F$, it
suffices to show that $\cM_1(\,f_\ast \phi u)(\lambda)$ has only a simple pole
at $\lambda = 0.$ We compute
\[
\begin{aligned}
	\lambda^2 \cM_1(\phi\,f_\ast u)(\lambda) &= \lambda^2\int_I \tau^{-\lambda}\,(f_\ast \phi\,u)(\tau)\,\frac{d\tau}{\tau}
	\\&= \lambda^2\int_{\pa X \times [0,1)^2} (t\,x)^{- \lambda} \phi(t,x)\,u(z,t,x)\,\frac{dt}{t}\frac{dx}{x}\,dV_{\pa X}
	\\&= \int_{\pa X} \int_{[0,1)^2} (t\,x)^{- \lambda} \pa_t\,\pa_x\bpns{\phi(t,x)\,u(z,t,x)}\,dx\,dt\,dV_{\pa X}
\end{aligned}
\]
using $f^\ast (\tau^{-\lambda}) = (t\,x)^{-\lambda}$, the identities $
\lambda (t\,x)^{-\lambda} = -x\pa_x (t\,x)^{-\lambda} = - t\pa_t
(t\,x)^{-\lambda}$, integration by parts (which valid in the region
$\re\lambda >> 0$) and analytic continuation. It follows that
\[
	\lim_{\lambda \smallto 0} \lambda^2 \cM_1(f_\ast \phi\,u)(\lambda) = \int_{\pa X} u(z,0,0)\,dV_{\pa X}
\]
so that vanishing of $\int_{\pa X} u \rst_{\pa X \times \set{0,0}}$ implies
simplicity of the pole at $\lambda = 0.$

By considering localizations at other boundary faces away from $\ftf \cap
\fzf$, (where the b-fibration is modeled on a simple projection), it is easily
seen that the only extended union (i.e.\ additional log term) contribution can
come from the corner $\ftf \cap \fzf.$ Thus, vanishing of the integral of $u
\rst_{\ftf \cap \fzf}$ is sufficient (and indeed necessary from the equivalence
of the Mellin characterization of polyhomogeneity) for $f_\ast u$ to have
leading index $(0,0) \in F.$

Once it is established that $\Tr(P) \in \cAphg^{(0,0)}(I) \subset C^0(I),$ it
is immediate that $\Tr(P)(0)$ consists of the pushforward from $\pi_I^{-1}(0) = \ftf
\cup \fzf \subset \tX$ and is given by the integral of the induced b-densities
thereon, from which \eqref{E:trace_zero} follows.
\end{proof}

\bibliographystyle{amsalpha}
\bibliography{references}

\end{document}